\documentclass[10pt,oneside]{amsart}

      \usepackage{amssymb}
      \usepackage{varioref}

      \theoremstyle{plain}
      \newtheorem*{theoremA}{Theorem A}
      \newtheorem*{theoremB}{Theorem B}
      \newtheorem*{theoremC}{Theorem C}
      \newtheorem*{corollaryD}{Corollory D}

      \newtheorem{theorem}{Theorem}[section]
      \newtheorem{lemma}[theorem]{Lemma}
      
      \newtheorem{proposition}[theorem]{Proposition}
      
      \theoremstyle{definition}
      \newtheorem{definition}[theorem]{Definition}
      
      \theoremstyle{remark}
      \newtheorem{remark}[theorem]{Remark}
      
      \labelformat{enumi}{(#1)}
      
\newcommand\norm[1]{\left\lVert#1\right\rVert}
\newcommand\sqnorm[1]{\left\lVert#1\right\rVert^2}
\newcommand\innerprod[2]{\ensuremath{\left\langle#1, #2\right\rangle}}

\newcommand\crosssection[1]{\ensuremath{\mathbf{\Gamma}\left(#1\right)}}

\newcommand\hess[1]{\operatorname{H}^{#1}}
\newcommand\ball[3][]{\ensuremath{B^{#1}_{#2}\left(#3\right)}}

\newcommand\pderivr{{\ensuremath{\partial_r}}}
\newcommand\pderivt{{\ensuremath{\partial_t}}}

\newcommand\algcurvop[1]{\ensuremath{\mathcal{C}_B(#1)}}
\newcommand\selfadjoint[1]{\ensuremath{S^2(#1)}}
\newcommand\intr[1]{#1^{\circ}}

\newcommand\ghost[1]{\phantom{\mathrel{#1}}}
\newcommand\rline[2]{\left.{#1}\right|_{#2}} 

\newcommand\set[2]{\ensuremath{\left\{ #1 \, | \, #2 \right\}}}

\newcommand\reals{\mathbb{R}}
\newcommand\complex{\mathbb{C}}

\newcommand\isocone{\ensuremath{C_{\text{iso} > 0}}}

\newcommand\piccone{\isocone}

\newcommand\vertproj[1]{#1^{\mathcal{V}}}
\newcommand\horzproj[1]{#1^{\mathcal{H}}}
\newcommand\vertdistr{\mathcal{V}}
\newcommand\horzdistr{\mathcal{H}}

\newcommand\nocompile[1]{}

\DeclareMathOperator\Hess{H}

\DeclareMathOperator\Scal{Scal}

\DeclareMathOperator\tr{tr}

\DeclareMathOperator\Ad{Ad}

\DeclareMathOperator\myspan{span}

\DeclareMathOperator\Or{O}

\DeclareMathOperator\sff{II}

\DeclareMathOperator\BigO{O}

      \makeatletter
      \def\@setcopyright{}
      \def\serieslogo@{}
      \makeatother

\begin{document}

   \author{Sebastian Hoelzel}
   \address{Mathematisches Institut, WWU M\"unster, Germany}
   \email{sebastian.hoelzel@uni-muenster.de}
 
   \title[Surgery]{Surgery stable curvature conditions}

   \begin{abstract}
   	 We give a simple criterion for whether a pointwise curvature condition is stable under surgery. Namely, a curvature condition $C$, which is understood to be an 
   	 open, convex, $\Or(n)$-invariant cone in the space of algebraic curvature operators, is stable under surgeries of codimension at least $c$ provided it contains 
   	 the curvature operator corresponding to $S^{c-1} \times \reals^{n-c+1}$, $c \geq 3$. 
   	 
   	 This is used to generalize the well-known classification result of positive scalar curvature in the simply-connected case in the following way: Any 
   	 simply-connected manifold $M^n$, $n \geq 5$, which is either spin with vanishing $\alpha$-invariant or else is non-spin admits for any $\epsilon > 0$ a metric
   	 such that the curvature operator satisfies $R > - \epsilon \norm{R}$. 
   \end{abstract}



   \thanks{The author was supported by SFB 878.}

   \date{\today}

\maketitle


\section{Introduction}
%
%

The aim of this work is to prove a sufficient criterium for a curvature condition to be stable under surgery and to exploit some of its consequences.

Given a 
smooth manifold $M^n$ with an $(n-k)$-dimensional sphere $S^{n-k}$ embedded with trivial normal bundle such that a tubular neighborhood of $S^{n-k}$ is diffeomorphic to $S^{n-k} \times D^k$, a \emph{surgery of codimension $k$} produces a new manifold via the following prescription:
\[
	\chi(M^n, S^{n-k}) := \Bigl[M^n \backslash (S^{n-k} \times D^k)\Bigr] \cup_{S^{n-k} \times S^{k-1}} \Bigl[ \overline{D}^{n-k+1} \times S^{k-1}\Bigr].
\]
\par
Consider the vector space $\algcurvop{\reals^n}$ of algebraic curvature operators satisfying the Bianchi identity. A subset $C \subset \algcurvop{\reals^n}$ will be called a \emph{curvature condition} if it is invariant under the natural $\Or(n)$-representation on $\algcurvop{\reals^n}$. We say that a Riemannian manifold $(M^n,g)$ \emph{satisfies $C$} provided for any linear isometry $\iota : \reals^n \to T_pM$ the pullback $\iota^*R(p) \in \algcurvop{\reals^n}$ of the curvature operator $R(p) \in \algcurvop{T_pM}$ of $(M,g)$ belongs to $C$. 
\par
This notion allows us to formulate 
\begin{theoremA} 
	Let $C \subset \algcurvop{\reals^n}$ be an open, convex $\Or(n)$-invariant cone with 
	\[
		R_{S^{c-1} \times \reals^{n-c+1}} \in C,
	\]
	for some $c \in \{3,\ldots, n \}$. Suppose $(M^n, g)$ is a Riemannian manifold satisfying $C$. Then a manifold 
	obtained from $M^n$ by performing surgery of codimension at least $c$ also admits a metric satisfying $C$.
\end{theoremA}

Here, $R_{S^{c-1} \times \reals^{n-c+1}} = \pi_{\bigwedge^2 \reals^{c-1}}: \bigwedge^2 \reals^n \to \bigwedge^2 \reals^n$ corresponds to the curvature operator of $S^{c-1} \times \reals^{n-c+1}$ equipped with its canoncial product metric. 
\par
We proceed with a couple of \emph{examples}, where the above theorem can be applied, which at the same time serve to illustrate the \emph{history} of surgery theorems in Riemannian geometry.
\begin{enumerate}
	\item \emph{Positive scalar curvature}. The set 
		\[
			C_{\Scal > 0} := \set{R \in \algcurvop{\reals^n}}{\tr(R) > 0}
		\]
		corresponds to the condition of positive scalar curvature in the usual sense. As it evidently contains $R_{S^d \times 
		\reals^{n-d}}$, if $d \geq 2$, Theorem A 
		yields stability of positive scalar curvature under surgery of codimension $\geq 3$, which 
		was first obtained by Gromov and Lawson and, independently, Schoen and Yau (see \cite{GromovLawson1} and \cite{SchoenYau1}, respectively). More 
		precisely, the former proved the surgery stability for $\Scal > 0$ as stated in Theorem A, 
		whereas the latter deduced a slightly more general result as 
		formulated in Theorem B 
		below concerning this specific condition by the use of some singular partial differential equations.
		
	\item \emph{Positive isotropic curvature}.
		Let $\pi \subset \reals^n \otimes \complex = \complex^n$  be a complex plane. The complex sectional curvature of $\pi$ of an 
		algebraic curvature tensor $R \in \algcurvop{\reals^n}$ is given by
		\[
			\sec(R)(\pi) := R_{\complex}(b_1, b_2, \overline{b_2}, \overline{b_1}),
		\]
		where $\{ b_1, b_2 \} \subset \pi$ is a unitary basis of $\pi$ and $R_{\complex}$ denotes the complex quadrilinear extension 
		of $R$. $\pi$ is called \textsl{isotropic}, if for all $v \in \pi$ 
		we have $\left(g_{\reals^n}\right)_{\complex}(v, v) = 0$, where $\left(g_{\reals^n}\right)_{\complex}$ denotes the complex 
		bilinear extension of $g_{\reals^n}$. \\
		Given a complex isotropic plane $\pi \subset \complex^n$, it is not hard to see that there exist orthonormal vectors
		$e_1, \ldots, e_4 \in \reals^{2n}$ such that $\pi = \myspan_{\complex} \{ e_1 + i e_2, e_3 + i e_4 \}$. Then, using
		the shortcut notation $R_{ijkl} := R(e_i, e_j, e_k, e_l)$, we get
		\[
			\sec(R)(\pi) = R_{1331} + R_{1441} + R_{2332} + R_{2442} - 2 R_{1234}.
		\]
		It readily follows that $R_{S^{n-1} \times \reals}$ is contained in the set
		\[
			\ghost{.........}\quad\piccone := \set{R \in \algcurvop{\reals^n}}{ \sec(R)(\pi) > 0 \text{ for all isotropic planes } \pi \subset \complex^n}
		\]
		of operators with positive isotropic curvature. Thus Theorem A 
		states the stability of the class of Riemannian manifolds with positive
		isotropic curvature under connected sum constructions, which recovers a theorem proved by Micallef and Wang (\cite{MicallefWang}) in 1993.

	\item \emph{Positive $p$-curvature}.	Consider
		\[
			C_{p > 0} := \set{R \in \algcurvop{\reals^n}}{ s_p(R)(P) > 0 \text{ for any $p$-plane } P \subset \reals^n },
		\]
		where $s_p(R)(P) := \sum_{j,k = p+1}^n R(e_j, e_k, e_k, e_j)$, with $e_{p+1}, \ldots, e_n$ being an orthonormal basis of $P^{\bot}$, is called the 
		\emph{$p$-curvature} of the plane $P$ with respect to the operator $R$. This is an open, convex, $\Or(n)$-invariant cone, and $R_{S^{d} \times \reals^{n-d}} 
		\in C_{p > 0}$, if and only if $d \geq p +2$, for this implies $\dim \left(\reals^{d} \cap P^{\bot}\right) \geq 2$. Thus Theorem A 
		gives 
		stability under surgery of codimension $\geq p+3$  for $C_{p \geq 0}$, which was proved by Labbi in 1997 (see \cite{Labbi1}) using the 
		construction method employed in \cite{GromovLawson1}.
		
	\item \emph{Pointwise almost nonnegative sectional curvature}.
		In a quite similar fashion, Sung proved in 2004 (see \cite{Sung}) that, for $\epsilon > 0$, the condition given by
		\[
			\tilde{C}_{\epsilon} := \set{ R \in \algcurvop{\reals^n}}{ \sec(R) > - \epsilon \Scal(R) }
		\]
		enjoys stability under surgery of codimension $\geq 3$, which is covered by Theorem A, 
		as obviously $R_{S^2\times\reals^{n-2}} \in 
		\tilde{C}_{\epsilon}$.
		\par $\tilde{C}_{\epsilon}$ contains the cone of curvature operators with nonnegative sectional curvature and converges (in the pointed Gromov-Hausdorff sense) to 
		this cone for $\epsilon \to 0$. 
		Therefore the family $\tilde{C}_{\epsilon}$ might reasonably be considered a condition of \emph{pointwise almost nonnegative sectional 
		curvature}.

	\item \emph{Positive $S$-curvature}.
		By identifying $\bigwedge^2 \reals^n$ with $\mathfrak{so}(n)$ and complexifying the latter to get $\mathfrak{so}(n, \complex)$, we can regard an operator  $R \in 
		\algcurvop{\reals^n}$ as an operator $R_{\complex} : \mathfrak{so}(n, \complex) \to \mathfrak{so}(n, \complex)$ through complex linear extension. For an 
		$\Ad_{\operatorname{SO}(n, \complex)}$-invariant subset $S \subset \mathfrak{so}(n, \complex)$ the definition 
		\[
			C(S) := \set{R \in \algcurvop{\reals^n}}{ \innerprod{R_{\complex}(X)}{\overline{X}}_{\complex} > 0 \text{ for all } X \in S}
		\]
		yields an open, convex $\Or(n)$-invariant cone (which moreover turns out to be Ricci flow invariant; see \cite{Wilking}). Recently, it was proved in 
		\cite{GurMaiSeshadri} that $C(S)$ is stable under connected sum constructions, if $S$ does not contain any elements of the form $v \wedge w$, with $v \in 
		\reals^n$, $w \in \complex^n$ and $\left(g_{\reals^n}\right)_{\complex}(v,w) = 0$. Since this latter condition is seen to be equivalent to the requirement 
		$R_{S^{n-1} \times \reals} \in C(S)$, this case also is covered by Theorem A.\\
		Furthermore, it was shown in \cite{GurMaiSeshadri} that $C(S) \subset \piccone$ and that $C(S_0) = \piccone$ can be achieved by taking 
		\[
			S_0 = \set{ X \in \mathfrak{so}(n, \complex)}{ \operatorname{rk}(X) = 2 \text{ and } X^2 = 0},
		\]
		thus it is not surprising that the proof in \cite{GurMaiSeshadri} consists in a generalization of the proof given 
		in \cite{MicallefWang}.
\end{enumerate}
\vspace{3mm}
\par
Theorem A 
follows from a more general result that we explain next. To fix notation, let $C \subset \algcurvop{\reals^n}$ be an open (and non-empty) curvature condition. We say \emph{$C$ satisfies an inner cone condition} with respect to an operator $S \in \algcurvop{\reals^n} \backslash \{ 0 \}$ if the following holds: For every $R \in C$ there is a $\rho = \rho(R) > 0$, depending continuously on $R$, such that 
\[
	R + C_{\rho} := \set{R+T}{T \in C_{\rho}} \subset C,
\]
where $C_{\rho}$ is an open, convex, $\Or(n)$-invariant cone containing $\ball{\rho}{S}$.\\
Note that if $C \subset \algcurvop{\reals^n}$ is an open, $\Or(n)$-invariant convex cone and $S \in C$, then $C$ automatically satisfies an inner cone condition with respect to $S$.

\begin{theoremB}
	Let $C \subset \algcurvop{\reals^n}$ be a curvature condition satisfying an inner cone condition with respect to 
	$R_{S^{n-k-1} \times \reals^{k+1}}$, for some $k \in \{0,\ldots, n-3\}$. Suppose $(M_i^n, g_{M_i})$, $i =1,2$, are two $n$-dimensional Riemannian manifolds 
	satisfying $C$ and let $N_i^l \subset M_i^n$ be closed $l$-dimensional submanifolds of $M_i$, with $0 \leq l \leq k$. 
	\par
	If there is an isomorphism  $\phi : \nu N_1 \rightarrow \nu N_2$ of the normal bundles of $N_i$ in $M_i$, then the \emph{joining of $M_1$ and $M_2$ along $\phi$} 
	defined by
	\[
		M_1 \#_{\phi} M_2 := \left(M_1 \backslash \ball{\epsilon}{N_1}\right) \cup_{\overline{\phi}} \left(M_2 \backslash \ball{\epsilon}{N_2}\right)
	\]
	also carries a metric satisfying $C$.  
\end{theoremB}

Here, $\overline{\phi}$ is given by $\exp \circ \phi \circ \left(\rline{\exp}{\nu^{\epsilon}N_1} \right)^{-1} :  
\partial\ball{\epsilon}{N_1} \stackrel{\approx}{\rightarrow} \partial\ball{\epsilon}{N_2}$ and $\epsilon > 0$ is meant to be chosen small enough such that the normal 
exponential mappings $\exp : \nu^{<2\epsilon} N_i \to \ball{2\epsilon}{N_i}$ are diffeomorphisms, $i = 1,2$, where $\nu^{<r} N_i := \set{\nu \in \nu N_i}{\norm{\nu} < r}$.
\par
This indeed implies Theorem A
, for $\chi(M^n, S^{n-k}) = M^n \sharp_{\phi} S^n$, with $\phi$ being the obvious isomorphism of the normal bundle of $S^{n-k} \subset M^n$ to the normal bundle of $S^{n-k} = S^n \cap \left(\reals^{n-k+1}\times\{0\}^{k}\right)\subset S^{n}$.
\par
Moreover, there are analogous results in the equivariant and conformally flat cases. These are outlined in sections \ref{sec:equivariant} and \ref{sec:conformallyflat}, respectively.
\vspace{3mm}
\par
Using the Ricci flow, B\"ohm and Wilking \cite{BoehmWilking2} proved that any Riemannian manifold with positive curvature operator is diffeomorphic to a spherical space form. Together with the work of Gallot and Meyer (cf. \cite{Petersen}) this yields a complete understanding of manifolds with nonnegative curvature operator. More precisely, a closed, simply connected Riemannian manifold with nonnegative curvature operator consists of a Riemannian product of manifolds which either are diffeomorphic to spheres, isometric to compact symmetric spaces or are K\"ahler manifolds biholomorphic to complex projective spaces. In particular, the class of closed, simply connected manifolds of a given dimension admitting a metric with nonnegative curvature operator consists of finitely many diffeomorphism types.
\par
In contrast to this, this rigidity result breaks down completely as soon as one tries to relax this curvature condition in the sense of

\begin{theoremC}
	Let $C \subset \algcurvop{\reals^n}$ be a curvature condition such that $C$ satisfies an inner cone condition with 	
	respect to any nonzero curvature operator with nonnegative eigenvalues (for instance, this holds, if $C$ is an open convex $\Or(n)$-invariant cone with $\{R \geq 0 
	\} \backslash \{ 0 \} \subset C$).\\ Suppose $M^n$, $n \geq 5$, is a closed, simply connected manifold. Then $M$ can be endowed 
	with a metric satisfying $C$, if either $M$ is non-spin, or $M$ is spin and $\alpha(M) = 0$.
\end{theoremC}

Here, in the spin-case, the mapping $M \mapsto \alpha(M)$ is a homomorphism $\Omega^{\text{Spin}}_* \to KO^{-*}(\text{pt})$, which coincides with a multiple of the $\hat{A}$-genus in dimensions $4k$.\\
This result was known in the case of positive scalar curvature by the combined work of Gromov-Lawson and Stolz (see \cite{GromovLawson1}, \cite{Stolz}). By using the same methods, Sung proved Theorem C for the special case of the conditions $\tilde{C}_{\epsilon}$ mentioned above (see \cite{Sung}). In fact, we show that these methods apply to the more general situation of Theorem C, where we employ a slightly generalized version regarding the vertical rescaling of a Riemannian submersion.
\par A particular case of Theorem C 
may be noted explicitly as

\begin{corollaryD}
	Let $M^n$ be as in Theorem C. 
	Then for any $\epsilon > 0$ there exists a metric $g_{\epsilon}$ on $M$ such that the curvature operator $R = 
	R_{(M^n, g_{\epsilon})}$ fulfills
	\[
		R > - \epsilon \norm{R}.
	\] 
	Here, $\norm{R}$ denotes the operator norm of $R$.
\end{corollaryD}

The paper is organized as follows. Theorems A and B are proved in section \ref{sec:mainproof}. Section \ref{sec:subm} is devoted to a simple submersion lemma which is used in section \ref{sec:class} to prove Theorem C. The last two sections deal with extensions of the surgery theorem to the equivariant and conformally flat cases, respectively.

\vspace{3mm}
\emph{Acknowledgements:} The author would like to thank Burkhard Wilking for his support during the preparation of this work, which contains the results of the author's Ph.D. thesis.


\section{Proofs of Theorem A and Theorem B}
\label{sec:mainproof}

The following theorem captures the main deformation procedure behind Theorem A and Theorem B. 

\begin{theorem}
	\label{thm:12}
	Let $C \subset \algcurvop{\reals^n}$ be a curvature condition satisfying an inner cone condition with respect to 
	$R_{S^{n-k-1} \times \reals^{k+1}}$. Let $(M^n, g_M)$ be a Riemannian manifold satisfying $C$ and suppose $N^k \subset M^n$ is a closed submanifold.
	Let $g_N$ be an arbitrary metric on $N^k$, $g_{\nu N}$ a vector bundle metric on $\nu N$ and $\nabla$ a connection on
	$\nu N$ being metric with respect to $g_{\nu N}$.
	\par
	Then for $\overline{r} > 0$ there is $\underline{r} \in (0, \overline{r})$ such that for every $r \in (0, \underline{r})$ there exists a complete metric $g_D$ on the 
	open manifold $D := M \backslash N$ with the following properties:
	\begin{enumerate}
		\item $g_D$ satisfies $C$.
		\item $g_D$ coincides with $g_M$ on $M \backslash D(\overline{r})$, where $D(\overline{r}) = \set{x \in M}{d_{g_M}(x, N) < \overline{r}}$
		\item In a neighborhood $U \subset M$ of $N$ the region $(U \backslash N, g_D)$ is isometric to 
			\[
				\left(\nu^rN \times (0, \infty), \rline{h}{\nu^rN} + dt^2 \right),
			\]
			where $\nu^rN := \set{v \in \nu N}{g_{\nu N}(v,v) = r^2}$ and $h$ is the connection metric determined by $g_N$, $g_{\nu N}$ and $\nabla$.
	\end{enumerate}
\end{theorem}

Theorems A and B 
follow immediately from Theorem \ref{thm:12} in conjunction with the following elementary property.

\begin{proposition}
	\label{prop:11}
	Suppose $C \subset \algcurvop{\reals^n}$ satisfies an inner cone condition with respect to $R_{S^{d} \times 
	\reals^{n-d}}$, $2 \leq d \leq n-1$. Then the same is true for $R_{S^{d+1} \times \reals^{n-d-1}}$.
\end{proposition}

\begin{proof}
	Suppose first that $C$ is an open, convex, $\Or(n)$-invariant cone with $R_d := R_{S^{d} \times \reals^{n-d}} \in C$. The orthogonal projection $\pi : 
	\reals^n \to \reals^{d+1}$ induces an inclusion $\algcurvop{\reals^{d + 1}} \subset \algcurvop{\reals^n}$. Now, $R_d = 
	\pi_{\bigwedge^2\reals^{d}}$ is contained in $\algcurvop{\reals^{d+1}}$, and so is $A \ast R_d$, if $A \in \Or(d+1) \subset \Or(n)$. This implies
	\[
		S := \int_{\Or(d+1)} A \ast R_d \, dm(A) \in \algcurvop{\reals^{d+1}}.
	\]
	Here, the standard Haar measure $m$ of $\Or(d+1)$ is used with the normalization $\int_{\Or(d+1)}dm = 1$. Because $S \in \algcurvop{\reals^{d+1}}$ is a fixed point 
	of the representation of $\Or(d+1)$, it follows easily that $S = \lambda R_{d+1}$ for some $\lambda > 0$, using the irreducible decomposition of this 
	representation. Furthermore, $S$ is contained in the convex hull $H$ of the orbit $\Or(d+1) \ast R_d$, because if $S \notin H$, we could find 
	an $A_0 \in H$ such that $d(A_0, S) = \inf_{A \in H} d(A, S) =: d(H, S) > 0$ due to the compactness of $H$. Convexity of $H$ implies then $\innerprod{S - 
	A_0}{A - A_0} \leq  0$ for each $A \in H$, which yields
	\[
		0 < \innerprod{S - A_0}{S - A_0} = \int_{\Or(d+1)} \innerprod{S - A_0}{A - A_0} \, dm(A) \leq 0,
	\]
	i.e. a contradiction. Hence convexity of $C$ gives us $\lambda R_{d+1} \in H \subset C$.
	\par
	Now, suppose $C$ merely satisfies an inner cone condition as in the statement. We can find a sequence of compact subsets $K_j$, with $K_j \subset \intr{K_{j+1}}$ and
	$C = \bigcup K_j$. By intersecting cones we get cones $C_j$ with $\ball{\rho_j}{R_d} \subset C_j$, where $\rho_j := \min_{R \in K_j} \rho(R)$, such that $R + 
	C_j \subset C$ for all $R \in K_j$. The above argument gives us numbers $\delta_j > 0$ with $\ball{\delta_j}{R_{d+1}} \subset C_j$ for each $j$. The function 
	\[
		\tilde{\delta}(R) := \max_{R \in K_j} \delta_j
	\]
	is positive on $C$ and it is easy to see that we can construct a continuous function
	\[
		\delta : C \to \reals^{>0}
	\]
	 with $\delta \leq \tilde{\delta}$ and $\ball{\delta(R)}{R + R_{d+1}} \subset R + C_j \subset C$ for all $R \in K_j \subset C$. This shows that $C$ satisfies an inner cone condition with 
	respect to $R_{d+1}$.
\end{proof}

We now turn to the proof of Theorem \ref{thm:12}.


\subsection{Setup and curvature formulas}

We are going to make use of a graph-like deformation procedure, which owes much to \cite{GromovLawson1}. The general setup is explained next.
\par
Without loss of generality $\overline{r} > 0$ can be assumed to be small enough that $\exp : \nu^{< 2\overline{r}} N \to M$ is an embedding. For some function $\theta : [0, \infty) \to \left[0, \frac{\pi}{2}\right]$, the mappings
\begin{equation}
	\label{eq:5}
		r(s) := \overline{r} - \int_0^s \cos \theta(u) \, du, \quad
		t(s) := \int_0^s \sin \theta(u) \, du,
\end{equation}
describe a curve $\gamma(s) := \left(r(s), t(s)\right)$ in the $(r,t)$-space, parametrized by arc length. The angle between the tangent $\gamma'(s)$ and $-\pderivr$ is then given by $\theta(s)$. The function $\theta(s)$ will be chosen in such a way that 
\begin{equation}
	\label{eq:10}
		\rline{t}{[0, \underline{s}]} \equiv 0 \text{ for some } \underline{s} > 0, \quad
		\rline{r'}{[\overline{s}, \infty)} \equiv 0 \text{ for some } \overline{s} > \underline{s} \text{ and } \quad r > 0
\end{equation}
hold.
\par
With the help of this 'model curve' we construct 
\[
	\gamma : \nu^{1}N \times \reals^{\geq 0} \to M \times \reals, \quad (\nu, s) \mapsto \left( \exp ( r(s) \nu ), t(s) \right),
\]
which we continue to call $\gamma$. Using $\gamma$ we deform the manifold $(M, g_M)$ to a new manifold $(D, g_D)$, defined by
\begin{equation}
	\label{eq:38}
	D := \set{\gamma(\nu, s)}{ \nu \in \nu^{1}N, s \geq 0} \cup \set{(p,0)\in M \times \reals}{d(p, N) \geq \overline{r}}
\end{equation}
and $g_D$ being the induced metric. Because of \eqref{eq:10}, $D$ will be indeed a smooth manifold diffeomorphic to $M \backslash N$.
\par
Our first task is to derive a useful formula for the curvature tensor of $D$. In order to do so, we aim at reexpressing the second fundamental form of $D$ in terms of known components.
\par
The derivative of $\gamma$ with respect to $s$ will be denoted by $\gamma'(\nu, s) = \frac{\partial \gamma}{\partial s}(\nu, s)$. Observe that
\begin{equation}
	\label{eq:6}
	\begin{split}
		\gamma'(\nu, s) &= d\exp (r(s)\nu)(r'(s)\nu) + t'(s) \rline{\pderivt}{\gamma(\nu, s)} \\
										&= r'(s)\rline{\pderivr}{\gamma(\nu, s)} + t'(s) \rline{\pderivt}{\gamma(\nu, s)} \\
										&= - \cos \theta(s) \rline{\pderivr}{\gamma(\nu, s)} + \sin \theta(s)\rline{\pderivt}{\gamma(\nu, s)}. 
	\end{split}
\end{equation}
We choose a local normal vector field $\mu(\nu, s)$ to $D$ by rotating $\gamma'(\nu, s)$ counterclockwise by $\frac{\pi}{2}$, thus getting
\begin{equation}
	\label{eq_3}
	\begin{split}
		\mu(\nu, s) &:= - r'(s) \rline{\pderivt}{\gamma(\nu, s)} + t'(s) \rline{\pderivr}{\gamma(\nu, s)}\\
				&\ghost{:}= \cos \theta(s) \rline{\pderivt}{\gamma(\nu, s)} + \sin \theta(s) \rline{\pderivr}{\gamma(\nu, s)}
	\end{split}
\end{equation}
In the following, we make use of the splitting of the tangential space of D given by
\[
	T_{\gamma(\nu, s)}D = \myspan \left\{ \gamma'(\nu, s) \right\} \oplus T_{\exp(r(s)\nu)}T\left(r(s)\right),
\]
where $T(r) := \set{ x \in M}{d(x, N) = r}$ denotes the distance tube of radius $r$ around $N$.

\begin{lemma}
	\label{lem:3}
	The second fundamental form of $D \subset M \times \reals$ with respect to the normal $\mu(\nu, s)$ at a point $\gamma(\nu, s)$, which is defined by 
	$\sff_D(v,w) = \innerprod{\nabla^D_v \mu}{w}$, is given by
	\begin{align*}
			\sff_D \left(\gamma', \gamma' \right) &= - \theta'(s), \\
			\sff_D \left(\gamma', v \right) &= 0, \\
			\sff_D \left(v, w \right) &= \sin \theta (s) \sff_{T(r(s))} \left(v, w \right) .
	\end{align*}
	Here, $v, w \in T_{\exp(r(s) \nu)} T(r(s))$ and $\sff_{T(r)}$ is the second fundamental form given by $\sff_{T(r)}(X,Y) = \sff_{T(r)}^{\nabla r}(X,Y) = 	
	- \innerprod{\nabla_XY}{\nabla r}$.
\end{lemma}

\begin{proof}
	From \eqref{eq:6} it follows
	\[
		\frac{\nabla^D \gamma'}{ds}(\nu, s) = \theta'(s) \left( \sin \theta(s) \rline{\pderivr}{\gamma(\nu, s)} + \cos 
				\theta(s) \rline{\pderivt}{\gamma(\nu, s)} \right) = \theta'(s) \mu(\nu, s),
	\]
	because for a fixed $\nu \in \nu^1N$ we have
	\[
		\frac{\nabla^D}{ds}\rline{\pderivt}{\gamma(\nu, s)} = 0 = \frac{\nabla^D}{ds}\rline{\pderivr}{\gamma(\nu, s)},
	\]
	since the two-dimensional submanifold
	$\set{ \left(\exp(r \nu), t \right) \in M \times \reals}{ r \in (0, \overline{r}), \, t \in \reals}$ is totally geodesic. Thus we obtain
	\[
		\sff_D \left(\gamma'(\nu, s), \gamma'(\nu, s) \right) = - g_D\left(\frac{\nabla^D \gamma'}{ds}(\nu, s), \mu(\nu, s)\right) = - \theta'(s).
	\]
	\par
	For the next equation we compute
	\[
		\frac{\nabla^D \mu}{ds}(\nu, s) = \theta'(s) \left( - \sin \theta(s) \rline{\pderivt}{\gamma(\nu, s)} + \cos \theta(s)
			\rline{\pderivr}{\gamma(\nu, s)} \right) = - \theta'(s) \gamma'(\nu, s)
	\]
	and deduce
	\[
		\sff_D \left(\gamma'(\nu, s), v \right) = g_D \left( \frac{\nabla^D \mu}{ds}(\nu, s), v \right) = 0.
	\]
	\par
	Finally, if we denote by $\tilde{v}$ and $\tilde{w}$ extensions of $v$ and $w$, respectively, to vector fields tangent to $M 
	\subset M \times \reals$, we get
	\begin{align*}
		\sff_D \left(v, w \right) &= - g_D \left( \nabla^M_{\tilde{v}}\tilde{w}, \mu(\nu, s) \right) \\
				&= - g_M \left( \nabla^M_{\tilde{v}}\tilde{w}, \sin \theta(s) \rline{\pderivr}{\gamma(\nu, s)} \right) \\
				&= \sin \theta (s) \sff_{T(r(s))} \left(v, w \right).
	\end{align*}
\end{proof}

We are going to compare the curvature tensor of $D$ at a point $\gamma(\nu, s)$ with the corresponding one of $M$ at $\exp(r(s)\nu)$. In order to do so,
we introduce the following notation for pulling back a given curvature operator to $\algcurvop{\reals^n}$.

\begin{definition}
	\label{def:007}
 	Fix once and for all an ordered $n$-tuple $E = (e_1, \ldots, e_n)$ of vectors $e_i \in \reals^n$ that form an orthonormal basis of $\reals^n$ and suppose an ordered 
 	$n$-tuple $B = (b_1, \ldots, b_n)$ of orthonormal vectors $b_j \in T_pM$ is given. Then the \emph{B-pullback} of $R_{(M, g)}(p) \in \algcurvop{T_pM}$ is given 
 	by $\iota^*R_{(M,g)}(p) \in \algcurvop{\reals^n}$, where $\iota : \reals^n \to T_pM$ is the linear map defined by $e_i \mapsto b_i$ for $i = 1,\ldots, n$.
\end{definition}
Given a $k$-tuple $B = (b_1, \ldots, b_k)$ and an $l$-tuple $C = (c_1, \ldots, c_l)$, $B + C$ will denote the $(k + l)$-tuple
$(b_1, \ldots, b_k, c_1, \ldots, c_l)$. \par
Now, the following choice of bases will prove useful: 
For $\nu \in \nu^1_qN \subset \nu^1N$ and $r \in (0, \overline{r})$ we get a canonically defined subspace $\horzdistr_{(\nu, r)} \subset T_{\exp(r\nu)}M \cap \left\{\pderivr\right\}^{\bot}$ by parallel translation of $T_qN$ along $t \mapsto \exp(t \nu)$, and an orthogonal complementary subspace $\vertdistr_{(\nu, r)}$ such that
\begin{equation}
	\label{eq:34}
	T_{\exp(r\nu)}T(r) =  \vertdistr_{(\nu, r)} \oplus \horzdistr_{(\nu ,r)}.
\end{equation}
$\vertdistr := \bigcup_{\nu, r > 0} \vertdistr_{(\nu ,r)}$ and $\horzdistr := \bigcup_{\nu, r > 0} \horzdistr_{(\nu ,r)}$ define smooth distributions on $D(\overline{r}) \backslash N$. For each $(\nu, r)$ we choose an orthonormal basis $V(\nu ,r) = (v_1, \ldots, v_{n-k-1})$ of $\vertdistr_{(\nu, r)}$ and an orthonormal basis $H(\nu, r) = (h_1, \ldots, h_k)$ of $\horzdistr_{(\nu, r)}$ and consider the pullback of the following curvature operators to $\algcurvop{\reals^n}$. Set $B(\nu ,r) := V + H$. Then denote
\begin{gather*}
	\text{the } \left(B(\nu ,r(s)) + \left(-\gamma'(\nu, s)\right)\right)\text{-pullback of } R_{(D, g_D)} \text{ by } \tilde{R}_D(\nu, s), \\
	\text{the } \left(B(\nu ,r) + \left(\rline{\pderivr}{\exp_q (r\nu)}\right)\right)\text{-pullback of }R_{(M,g_M)}\text{ by }\tilde{R}_M(\nu, r)\text{ and}\\ 
	\text{the } \left(B(\nu ,r) + \left(\pderivt\right)\right)\text{-pullback of } R_{(T(r) \times \reals, \rline{g_M}{T(r)} +dt^2)}\text{ by }\tilde{R}_T(\nu, r).
\end{gather*}
Note that the dependences of the pullback operators on the chosen bases are suppressed by the notation and that, in fact, the mappings $(\nu, r) \mapsto \tilde{R}_*(\nu, r)$ are not even continuous, as in general there do not exist global sections of $\horzdistr$ and $\vertdistr$. However, any two choices of bases merely result in a change caused by a transformation by an element of $\Or(n)$. This does not matter as long as all tensors are being pulled back by a consistent choice of bases as just described since we are dealing with $\Or(n)$-invariant subsets of $\algcurvop{\reals^n}$ and are only interested in whether or not those pullbacks are contained in these sets.

\begin{proposition} 
	\label{prop:1}
	With the above notation, the equation
	\[
		\tilde{R}_D \left(\nu, s \right) = \cos^2 \theta(s) \, \tilde{R}_M \left(\nu, r(s) \right) + \sin^2 \theta(s) \, 
				\tilde{R}_T \left(\nu, r(s) \right)+ E \left(\nu, s \right)
	\]
	holds on $\nu^1N \times (0, \infty)$. $E : \nu^1N \times (0, \infty) \rightarrow \algcurvop{\reals^n}$ satisfies the 
	estimate
	\[
		\norm{E(\nu, s)} \leq \cos \theta(s) \left( 1 - \cos \theta(s) \right) C_1 + \frac{\theta'(s) \sin \theta(s) }{r(s)} C_2,
	\]
	where the constants $C_1$, $C_2$ depend only on $(D(\overline{r}), g_M)$ and $N$. 
\end{proposition}
In order to prove this, we will need the following fact about the second fundamental form of $T(r)$ (cf. chapter 5 of \cite{Eschenburg}).
\begin{lemma}
	\label{lem:2}
	Let $T(r) \subset M$ be a distance tube around a closed submanifold $N^k \subset M^n$. Then there exists a mapping $A : \nu^1 N \times [0, \overline{r}) \to 
	\selfadjoint{T^*M}$, $A(\nu, r) \in \selfadjoint{T^*_{\exp(r\nu)}M}$, such that
	\[
		\sff_{T(r)} \left(v, w \right) = \frac{1}{r} \pi_{\vertdistr}^{\flat}(v, w) + A(\nu, r) (v, w),
	\]
	where $r \in (0, \overline{r})$, $\nu \in \nu^1N$, and $v,w \in T_{\exp(r\nu)}M \cap \{\nabla r \}^{\bot}$ ($\pi_{\vertdistr_{(\nu, r)}}$ denotes orthogonal 
	projection onto $\vertdistr_{(\nu, r)} \subset T_{\exp(r\nu)}M$.). Moreover, we have $\norm{A(\nu, r)} \leq C$ for some constant and
	\[
		\rline{A(\nu, 0)}{T_xN \times T_xN} = \sff_N^{\nu}, \quad \rline{A(\nu, 0)}{\nu_xN \times T_xN} = 0, \quad \rline{A(\nu, 0)}{\nu_xN \times \nu_xN} = 0,
	\]
	if $\nu \in \nu_xN$.
\end{lemma}

\begin{proof}
	Consider $h := \frac{1}{2}r^2$, where $r(p) = d(p, N)$. Since we have $h = \frac{1}{2}\sum_{j=k+1}^n (x^j)^2$ in Fermi coordinates, where $\{x^1=\ldots =x^k = 0\} 
	\subset N$, this is a smooth function on $D(\overline{r})$ and has vanishing differential on $N$. Thus, the hessian of $h$ is readily computed to be $\hess{h}(x) = 
	\pi_{\nu_xN}^{\flat}$ at $x \in N$. Moreover, the corresponding coordinate expressions immediately show
	\[
		\hess{h} \circ \exp(r \nu) = \pi_{\tilde{\vertdistr}_{\exp(r\nu)}}^{\flat} + r \tilde{A}(\nu, r), 
	\]
	where $\tilde{\vertdistr}_{\exp(r\nu)}$, given by parallel transport of $\nu_x N$ along $t \mapsto \exp(t \nu)$, with $\nu \in \nu^1_xN$, defines a smooth 
	distribution on $D(\overline{r})$ and $\tilde{A}$ is a bounded continuous map $\nu^1 N \times [0, \overline{r}) \to \selfadjoint{T^*M}$. Recalling that the second 
	fundamental form of $T(r)$ is given by $\sff_{T(r)} = \hess{r}$ and
	\[
		\hess{\frac{1}{2}r^2} = dr^2 + r \hess{r},
	\]
	the first claims of the lemma follow by restricting the above equation to $\{\partial_r\}^{\bot}$ and observing that $\vertdistr_{(\nu, r)} = 
	\tilde{\vertdistr}_{\exp(r\nu)} \cap \{\partial_r\}^{\bot}$. \par
	For the last claim, let $v \in \nu^1_xN \cap \{\nu\}^{\bot}$, $h \in T^1_xN$, and consider Jacobi fields $J_v$ and $J_h$ along $t \to \exp(t\nu)$ with initial 
	conditions $J_v(0) = 0$, $J_v'(0) = v$ and $J_h(0) = h$, $J_h'(0) = \sff_N^{\nu}(h,\cdot)^{\sharp} \in T_xN$, respectively. Recall that 
	$\sff_{T(r)}(J,J)=\innerprod{J'}{J}$, then
	\[
		A(\nu,0)(h,h) = \lim_{t\to 0}A(\nu, t)(J_h(t), J_h(t)) = \lim_{t \to 0} \innerprod{J_h'(t)}{J_h(t)} - \frac{\pi_{\vertdistr}^{\flat}(J_h(t), J_h(t))}{t},
	\]
	which equals $\sff_N^{\nu}(h,h)$, since $\norm{\pi_{\vertdistr}(J_h(t))} = \BigO(t^2)$. Observing that $\frac{J_v(t)}{\norm{J_v(t)}} \to v$, we derive
	\[
		A(\nu, 0)(h,v) = \lim_{t\to 0} \innerprod{J_h'(t)}{\frac{J_v(t)}{\norm{J_v(t)}}} - \frac{1}{t}\pi_{\vertdistr}^{\flat}\left(J_h(t), 
			\frac{J_v(t)}{\norm{J_v(t)}}\right) = 0.
	\]
	Finally, we have
	\[
		A(\nu, 0)(v,v) = \lim_{t \to 0} \frac{\innerprod{J_v'(t)}{J_v(t)}}{\innerprod{J_v(t)}{J_v(t)}} - \frac{1}{t}
	\]
	which delivers the sought conclusion by repeated application of l'Hospital's rule.
\end{proof}

\begin{proof}[Proof of Proposition \ref{prop:1}]
	By restricting all arguments to $T_{\exp(r(s)\nu)}T(r) \subset T_{\gamma(\nu, s)}D$ we obtain, using the Gau\ss~equation twice and lemma \ref{lem:3},
	\begin{align*}
		\rline{R_D}{\gamma(\nu, s)} &= \rline{\left(R_{M \times \reals} + \sff_D \wedge \sff_D 
							\right)}{\gamma(\nu, s)} \\
				&= \rline{\left(R_M + \sin^2 \theta(s) \sff_{T(r)} \wedge \sff_{T(r)} \right)}{\exp(r(s)\nu)} \\
				&= \rline{\left( \cos^2 \theta(s) R_M + \sin^2 \theta(s) R_{T(r)} \right)}{\exp(r(s)\nu)}. 
	\end{align*}
	By switching to the algebraic identifications made above, we thus immediately get
	\begin{align*}
		\tilde{R}_D(\nu, s) &= 
					\left( \cos^2 \theta(s) \tilde{R}_M + \sin^2 \theta(s) \tilde{R}_T \right)(\nu, r(s)),
	\end{align*}
	where both sides are to be restricted to $\reals^{n-1} \times \{ 0 \} \subset \reals^n$.
	\par Next, using \eqref{eq:6} and the Gau\ss~equation we compute, with $v_i \in T_{\exp(r(s)\nu)}T(r) \allowbreak\subset T_{\gamma(\nu, s)}D$,
	\begin{align*}
		\rline{R_D}{\gamma(\nu, s)}\left(v_1, \gamma', \gamma', v_2\right) &= \rline{\left(R_{M \times \reals} + \sff_D \wedge 
							\sff_D \right)}{\gamma(s, \nu)} \left(v_1, \gamma', \gamma', v_2\right) \\
					&= \cos^2 \theta (s) \rline{R_M}{\exp(r(s)\nu)} \left(v_1, \pderivr, \pderivr, v_2\right) \\
					&\ghost{=} - \theta'(s) \sin \theta(s) \rline{\sff_{T(r)}}{\exp(r(s)\nu)} (v_1, v_2),
	\end{align*}
	since $\sff_D(v_i, \gamma') = 0$.	Expressed in the algebraic setting, this simply means, with $\tilde{v}_1, \tilde{v}_2 \in \reals^{n-1} \subset \reals^n$,
	\begin{align*}
		\tilde{R}_D(\nu, s) (\tilde{v}_1, e_n, e_n, \tilde{v}_2) &= \cos^2 \theta(s) \tilde{R}_M(\nu, r(s)) (\tilde{v}_1, e_n, e_n, 
					\tilde{v}_2) \\ 
			&\ghost{=} + \theta'(s) \sin \theta (s) \left( \frac{1}{r(s)}\pi_{\reals^{n-k-1}}^{\flat}(\tilde{v}_1, \tilde{v}_2) + \BigO(1) \right), 
	\end{align*}  
	\par Finally, computations done in the same manner yield
	\begin{align*}
		\tilde{R}_D(\nu, s) (\tilde{v}_1, \tilde{v}_2, \tilde{v}_3,e_n) &= - \rline{R_D}{\gamma(\nu, s)}(v_1, v_2, v_3, \gamma') \\
				&= - \rline{\left ( R_{M \times \reals} + \sff_D \wedge \sff_D \right)}{\gamma(\nu, s)} (v_1, v_2, v_3, \gamma') \\
				&= \cos \theta(s) \rline{R_M}{\exp(r(s)\nu)}\left(v_1, v_2, v_3, \pderivr\right) \\
				&\ghost{=} - \left( \sff_D(v_1, \gamma') \sff_D(v_2, v_3) - \sff_D(v_1, v_3) \sff_D(v_2, \gamma') \right) \\
				&= \left(\cos^2 \theta(s) + \cos \theta (s) \left( 1- \cos \theta (s)\right)\right) \\
				&\ghost{=} \cdot \tilde{R}_M (\nu, r(s)) (\tilde{v}_1, \tilde{v}_2, \tilde{v}_3, e_n).
	\end{align*}
	This finishes the proof.
\end{proof}

The occurence of $\tilde{R}_T$ in the above formula will play a key role, because of the following theorem.

\begin{theorem}
	\label{thm:1}
	Let $C \subset \algcurvop{\reals^n}$ be a curvature condition satisfying an inner cone condition with respect to 
	$R_{S^{n-k-1} \times \reals^{k+1}}$. 
	Let $N^k \subset M^n$ be a closed submanifold of the Riemannian manifold $(M^n, g_M)$.
	\par Then there exists $r_* > 0$ such that for all $r \in (0, r_*)$ the Riemannian manifold $\left(T(r) \times \reals, g_{T(r)} + g_{\reals}\right)$ satisfies $C$.
	Moreover, there exists $L > 0$ such that for a $\left(V+H+\left(\pderivt\right)\right)$-pullback $\tilde{R}_T(\nu,r)$ of $R_{(T(r) \times \reals, 
	\rline{g_M}{T(r)} + dt^2)}$ we 
	have
	\[
		\tilde{R}_T (\nu ,r) \in \ball{\frac{L}{r}}{R_{S^{n-k-1}(r) \times \reals^{k+1}}} \subset C.
	\]
\end{theorem}

Of course, $r_*$ is meant to be chosen so small that $T(r)$ is an embedded submanifold of $M$ for all $r < r_*$.

\begin{proof}
	Lemma \ref{lem:2} implies
	\[
		\sff_{T(r)} \wedge \sff_{T(r)} = \frac{1}{r^2} \pi_{\vertdistr}^{\flat} \wedge \pi_{\vertdistr}^{\flat} + \frac{2}{r} \pi_{\vertdistr}^{\flat} \wedge A(\nu, r) + 
			\BigO(1) 
	\]
	Hence, together with the Gau\ss~equation 
	\[
		R_{T(r)} = \rline{R_{(M, g_M)}}{T_pT(r)} + \sff_{T(r)} \wedge \sff_{T(r)}
	\]
	we get
	\begin{equation}
		\label{eq:35}
		\tilde{R}_T(\nu, r) = \frac{1}{r^2} R_{S^{n-k-1} \times \reals^{k+1}} + E(\nu, r),
	\end{equation}
	upon pulling back to $\reals^n$, with some tensor $E(\nu, r) \in \algcurvop{\reals^n}$ satisfying $\norm{E(\nu, r)} \leq Lr^{-1}$, where $L$ does not depend on the 
	choice of pullback. 
	\par
	There exists an $\Or(n)$-invariant open cone $\tilde{C} \subset \algcurvop{\reals^n}$ with $S := R_{S^{n-k-1} \times \reals^{k+1}} \in \tilde{C}$ and a 
	compact set $K \subset \algcurvop{\reals^n}$ such that $\tilde{C} \backslash K \subset C$,
	i.e. $C$ contains a truncated cone in whose interior $\lambda S$ can be found, for all $\lambda > 0$ sufficiently large. Indeed, take 
	an arbitrary $R \in C$. 
	Then $R + C_{\rho} \subset C$ implies $\ball{\mu \rho}{R + \mu S} \subset C$ for $\mu > 0$. Therefore we can find a $\lambda_0 > 0$ 
	such that $\lambda S \in 
	C$ for $\lambda \geq \lambda_0$. Due to the inner cone condition, we then have $S + C_{\rho'} \subset C$ for some $\rho' > 0$, and we can find an open cone 
	$C'$ such that $C' \backslash \left(S + C_{\rho'}\right)$ is bounded. Then $\tilde{C} = \Or(n) \ast C'$ has the desired properties.
	\par
	Now, given $\rho > 0$ with $\ball{\rho}{R_{S^{n-k-1} \times \reals^{k+1}}} \subset \tilde{C}$ and it follows that
	\[
		\ball{\frac{\rho}{r^2}}{R_{S^{n-k-1}(r) \times \reals^{k+1}}} \subset \tilde{C},
	\]
	for if $S \in \ball{\frac{\rho}{r^2}}{R_{S^{n-k-1}(r) \times 	\reals^{k+1}}}$, then we have
	$\norm{r^2 S - R_{S^{n-k-1} \times \reals^{k+1}}} < \rho$, i.e. $r^2 S \in \ball{\rho}{R_{S^{n-k-1} \times \reals^{k+1}}}\subset \tilde{C}$. Since 
	$\tilde{C}$ is a cone, we have $S \in \tilde{C}$. Moreover, for some $\tilde{r} > 0$, even 
	\begin{equation}
		\label{eq:36}
		\ball{\frac{\rho}{r^2}}{R_{S^{n-k-1}(r) \times \reals^{k+1}}} \subset \tilde{C} \backslash K \subset C
	\end{equation}
	holds, if $r \in (0, \overline{r})$.
	\par
	Then, for $r \in (0, r_*)$, $r_* := \min\left\{\tilde{r}, \frac{\rho}{L} \right\}$, the inequality $\frac{L}{r} < \frac{\rho}{r^2}$ holds, hence \eqref{eq:35} 
	implies
	\[
		\tilde{R}_T(\nu, r) \in \ball{\frac{L}{r}}{R_{S^{n-k-1}(r) \times \reals^{k+1}}} \subset \ball{\frac{\rho}{r^2}}{R_{S^{n-k-1}(r) \times \reals^{k+1}}} \subset 
			C.
	\]
	This completes the proof.
\end{proof}

Using entirely analogous arguments we can also prove the following variant.

\begin{proposition}
	\label{prop:2}
	Let $C \subset \algcurvop{\reals^n}$ be a curvature condition satisfying an inner cone condition with respect to 
	$R_{S^{n-k-1} \times \reals^{k+1}}$. 
	Let $N^{k+1} \subset M^{n+1}$ be a compact submanifold of a Riemannian manifold $(M^{n+1}, g_M)$, with totally geodesic boundary $\partial M$, $\partial N \subset 
	\partial M$ and $\nu_qN \subset T_q \partial M$ for $q \in \partial N$, such that the normal exponential map is defined on $\nu^{< \epsilon}N$, for some 
	$\epsilon > 0$. 
	\par	
	Then there exists $r_* \in (0, \epsilon)$ such that for all $r \in (0, r_*)$ the Riemannian manifold $\left(T(r), g_{T(r)}\right)$ satisfies $C$.
\end{proposition}


\subsection{Construction of the deformed metric}

In this section we describe a construction of the new metric $g_D$.
This consists mainly in carefully prescribing the angular function $\theta(s)$ which needs to be done in a way maintaining
the given curvature condition. To that end, we first choose appropriate radius parameters which ensure the validity of the estimates given above. The construction itself will be subdivided into three steps.
\par
In summary, the goal of the first two steps consists in constructing a monotone increasing function $\theta : [0,\infty) \rightarrow [0, \frac{\pi}{2}]$, which starts out of zero and reaches $\frac{\pi}{2}$ in finite time, whereas the third step is necessary to smooth out the
metric which at that time will already be given as the product of the metric induced by a distance tube and a real line.
\par
By restricting our attention to the compact region 
\[
	D(\overline{r}) = \set{p \in M}{d(p,N) \leq \overline{r}}
\]
we can find a compact $\Or(n)$-invariant set $K \subset C$ such that $(D(\overline{r}), g_M)$ satisfies $K$.
Then the inner cone condition with respect to $R_{S^{n-k-1} \times \reals^{k+1}}$ implies the existence of a number $\rho > 0$
such that $R + C_{\rho} \subset C$ for all $R \in K$, where $C_{\rho}$ contains $\ball{\rho}{R_{S^{n-k-1} \times \reals^{k+1}}}$. This in turn implies
\begin{equation}
	\label{eq_2}
	\ball{\frac{\rho}{\lambda^2}}{R + R_{S^{n-k-1}(\lambda) \times \reals^{k+1}}} \subset C
\end{equation}
for $\lambda > 0$.
\par
The starting radius of the bending process $r_S$ is chosen in order to fulfill
\begin{equation}
	\label{eq:11}
	r_S < \min\left\{1, r_*, \frac{\rho}{4L}, \frac{\rho^{\frac{1}{2}}}{2} \left(\sup_{D(\overline{r})}\norm{R_M} + C_1 
		\right)^{-\frac{1}{2}} ,\overline{r} \right\}. 
\end{equation}
Here, the constants $L$ and $r_*$ are those given by Theorem \ref{thm:1}, $C_1$ by Proposition \ref{prop:1}.


\subsection{Step 1: Initial bending}

This first step consists in a slight increase of the bending angle, beginning from zero and reaching
an arbitrary small, but positive angle.

\begin{lemma}
	There exists $s_0 > 0$ and a non-decreasing smooth function $\theta : [0, s_0] \to [0, \theta_0]$ with $\theta$ being constant in neighborhoods of $0$ and $s_0$ with
	values $0$ and $\theta_0 > 0$, respectively, such that $\tilde{R}_D(\nu, s) \in C$ and $r(s) > 0$ for $s \in [0, s_0]$.
\end{lemma}

\begin{proof}
	There exists $\epsilon > 0$ such that for all $r \in \left(\frac{r_S}{2}, \overline{r}\right)$ and $\nu \in \nu^1N$ we have
	\begin{equation}
		\label{eq:7}
		\ball{\epsilon}{\tilde{R}_M(\nu, r)} \subset C.
	\end{equation}
	\par
	Next we choose $\theta_0 \in \left(0, \frac{r_S}{8}\right)$ so small that the following conditions hold:
	\begin{equation}
		\label{eq:12}
			\begin{gathered}
			\sin^2 \theta_0 \left( \sup_{p \in S} \norm{R_M(p)} + \sup_{p \in S} \norm{R_{T\left(d(p,N)\right)}(p)} \right) 
						< \frac{\epsilon}{2}, \\ 
				\left( 1 - \cos \theta_0 \right) C_1 + \frac{2\sin \theta_0}{r_S} C_2 < \frac{\epsilon}{2}. 
			\end{gathered}
	\end{equation}
	\par
	We set $s_0 := \overline{r} - \frac{r_S}{2}$ and prescribe $\theta(s)$ on the initial interval $[0, s_0]$ via
	\[
			\theta(s) := \begin{cases}
				0, &\text{for } s \in \left[0, \overline{r} - r_S\right], \\
				\theta_0, &\text{for } s \in \left[\overline{r} - \frac{3}{4}r_S, s_0 \right],
			\end{cases}
	\]
	such that $\theta'(s) \in [0,1]$. Now, with the help of Proposition \ref{prop:1} and \eqref{eq:12} we see that for $s \in [0,s_0]$
	\begin{align*}
		\norm{\tilde{R}_D(\nu, s) - \tilde{R}_M(\nu, r(s))} &\leq \sin^2 \theta(s) \rline{\left( \norm{R_M} + \norm{R_{T(r(s))}} \right)}{(\exp(r(s)\nu))} \\
					&\ghost{\leq} + \cos \theta(s) \left( 1 - \cos \theta(s) \right) C_1 + \frac{\theta'(s) \sin \theta(s)}{r(s)} C_2 \\
					&< \epsilon.
	\end{align*}
	In the last inequality we used that  $r(s) > \overline{r} -  s_0 = \frac{r_S}{2}$ for $s \in [0, s_0]$, which follows from \eqref{eq:5} and the construction of 
	$\theta$. The upper inequality combined with \eqref{eq:7} implies $\tilde{R}_D(\nu, s) \in C$ for $s \in [0, s_0]$, $\nu \in \nu^1N$.
\end{proof}


\subsection{Step 2: Inductive increasing of the bending angle}
In order to extend $\theta$, while keeping $\tilde{R}_D$ in $C$ we use the decomposition given by Proposition \ref{prop:1} and the fact that $C$ satisfies the appropriate inner cone condition.

\begin{lemma}
	\label{lem:step2}
	There exists $r^{*} \in (0, \overline{r})$ such that for every $r \in (0, r^{*})$ there is an extension of $\theta$ to a smooth non-decreasing
	function $\theta : [0, \infty) \to \left[0, \frac{\pi}{2}\right]$ such that $\tilde{R}_D(\nu, s) \in C$, $r(s) > 0$ for $s \geq 0$ as well as 
	$\rline{\theta}{[\overline{s}, \infty)} \equiv \frac{\pi}{2}$ and $\rline{r}{[\overline{s}, \infty)} \equiv r$ for some $\overline{s} > 0$ big enough.
\end{lemma}

\begin{proof}
		In fact, because of \eqref{eq_2} it will be sufficient to show
		\[
			\tilde{R}_D(\nu, s) \in \ball{ \rho \frac{\sin^2 \theta(s)}{r(s)^2}}{\tilde{R}_M(\nu, r(s)) + R_{S^{n-k-1}\left(\frac{r(s)}{\sin \theta(s)}\right) \times 
				\reals^{k+1}}}
		\]
		for maintaining $\tilde{R}_D \in C$, which of course is equivalent to
		\begin{equation}
			\label{eq:37}
			\norm{\tilde{R}_D(\nu, s) - \left(\tilde{R}_M(\nu, r(s)) + R_{S^{n-k-1}\left(\frac{r(s)}{\sin \theta(s)}\right) \times \reals^{k+1}} \right)}
			< \rho \frac{\sin^2 \theta(s)}{r(s)^2}.
		\end{equation}
		Therefore, we can estimate employing Proposition \ref{prop:1} 
		\begin{align*} 
			&\norm{\tilde{R}_D(\nu, s) - \left(\tilde{R}_M(\nu, r(s)) + R_{S^{n-k-1}\left(\frac{r(s)}{\sin \theta(s)}\right) \times \reals^{k+1}} \right)} \\
			&\ghost{xx}\leq \sin^2 \theta(s) \left( \norm{\tilde{R}_M(\nu, r(s))} + \norm{ \tilde{R}_T(\nu, r(s)) - R_{S^{n-k-1}(r(s)) \times 
						\reals^{k+1}}} \right) + \norm{E(\nu, s)} \\
			&\ghost{xx}\leq \sin^2 \theta(s) \left( \sup_{p \in D(\overline{r})} \norm{R_M(p)} + \frac{L}{r(s)}\right) + \cos \theta (s)\left(1 - 
						\cos \theta (s)\right) C_1 \\
			&\ghost{xx}\ghost{\leq} + \frac{\theta'(s) \sin \theta(s)}{r(s)} C_2,
		\end{align*}
		where we used Theorem \ref{thm:1} in the last inequality.
		Due to \eqref{eq:11}, we readily estimate \setlength\multlinegap{60pt}
		\begin{multline*}
			\sin^2 \theta(s) \sup_{p \in D(\overline{r})} \norm{R_M(p)}  + \cos \theta (s)\left(1 - \cos \theta (s)\right) C_1 \\
			\leq \sin^2 \theta(s) \left( \sup_{p \in D(\overline{r})} \norm{R_M(p)} + C_1 \right) < \sin^2 \theta(s) \frac{\rho}{4r(s)^2}
		\end{multline*}
		and
		\[
			\frac{L}{r(s)} < \frac{\rho}{4r(s)^2}.
		\]
		Hence \eqref{eq:37} holds, provided
		\[
			\frac{\theta'(s) \sin \theta(s)}{r(s)} C_2 \leq \sin^2 \theta(s) \frac{\rho}{2r(s)^2},
		\]
		which is equivalent to
		\begin{equation}
			\label{eq_1}
			\theta'(s) \leq \frac{\rho}{2C_2} \frac{\sin \theta(s)}{r(s)}.
		\end{equation}
		The remaining task is to extend the function $\theta$ from  $[0, s_0]$, to the interval $[0, \infty)$ 
		while maintaining inequality \eqref{eq_1}, the condition $r(s) > 0$ and $\rline{\theta}{[\overline{s}, \infty)} \equiv \frac{\pi}{2}$ for some $\overline{s} > 0$ 
		big enough.
		\par
		One possible way of doing this is to prescribe this extension of $\theta$ inductively as follows (but cf. also \cite{RosenbergStolz}). Suppose $\theta$ is already 
		defined on $[0, s_l]$. Set
		\[
			\theta_l := \theta(s_l) \in \left[\theta_0, \frac{\pi}{2}\right), \quad r_l := r(s_l) > 0
		\]
		(if $\theta_l = \frac{\pi}{2}$, we are done) and define $s_{l+1} := s_l + \frac{r_l}{2}$. We construct a smooth function $\eta_l : \reals \rightarrow \left[0, 
		\frac{\rho}{4C_2} \frac{\sin \theta_l}{r_l}\right]$ matching the following requirements:
		\[
			\eta_l \equiv \begin{cases}
				0, &\text{on } \left[s_l, s_l + \frac{r_l}{16}\right], \\
				\frac{\rho}{4C_2} \frac{\sin \theta_l}{r_l}, &\text{on } \left[s_l + \frac{r_l}{8} , s_{l+1} - \frac{r_l}{8}\right], \\
				0, &\text{on } \left[s_{l+1} - \frac{r_l}{16}, s_{l+1}\right]. \\
			\end{cases}
		\]
		Using this bump function we extend $\theta$ smoothly to the interval $\left[0, s_{l+1}\right]$ by setting for $s \in (s_l, s_{l+1})$
		\[
			\theta(s) := \theta(s_l) + \int_{s_l}^s \eta_l(u) \, du.
		\]
		Then, because of \eqref{eq:5},
		\[
			r_l \geq r(s) \geq \frac{r_l}{2} > 0
		\]
		and inequality \eqref{eq_1} is fulfilled, since
		\[
			\theta'(s) \leq \frac{\rho}{4C_2} \frac{\sin \theta_l}{r_l} < \frac{\rho}{2C_2} \frac{\sin \theta(s)}{r(s)}.
		\]
		Furthermore, we obtain the decisive estimate
		\[
			\theta_{l+1} - \theta_l := \theta(s_{l+1}) - \theta(s_l) \geq \int_{s_l + \frac{r_l}{8}}^{s_{l+1} - \frac{r_l}{8}} \eta_l(u)\,du \geq \frac{\rho}{16C_2} \sin 
				\theta_0.
		\]
		Therefore the amount of growth of $\theta$ in the interval $[s_l, s_{l+1}]$ is bounded from below independently of its length. This shows that the target value 
		$\theta = \frac{\pi}{2}$ can be reached by finitely many, say $m$, such bends. Of course, for the last bend the function $\eta_m$ has to be adjusted slightly in 
		order to avoid values above $\frac{\pi}{2}$.
		Moreover, any $r < r^{*} := r(s_m)$ can be achieved by inserting a straight line segment before performing the last bending step.
		\par
		Finally, we extend $\theta$ to $[0, \infty)$ by setting $\theta \equiv \frac{\pi}{2}$ on $[s_m, \infty)$.
\end{proof}


\subsection{Step 3: Smoothing of the end}
\label{sec:7}

\begin{lemma}
	\label{lem:step3}
	There exists $r^{**} > 0$ such that for any $r \in (0, r^{**})$ there exists a metric $g(t) + dt^2$ on $T(r) \times [t(\overline{s}), 
	\infty)$ which coincides with the induced metric of $M \times \reals$ for $t \in [t(\overline{s}), t(\overline{s})+1]$, which equals $\rline{h}{\nu^rN} + dt^2$ for 
	$t \geq t(\overline{s})+2$ and such that $C$ is satisfied.
\end{lemma}

Here, $h$ is the connection metric given by the prescribed metrics on $N$, $\nu N$ and the connection $\nabla$.

\begin{proof}
	On $B = D^{g_M}(\overline{r})$ we construct a one-parameter family $g(t)$, $t \in [0,1]$, with
	\[
		g(t) = \begin{cases}
			\rline{g_M}{B}, &\text{for } t \in \left[0, \frac{1}{4}\right], \\
		  \rline{h}{B}, &\text{for } t \in \left[\frac{3}{4}, 1\right]. \\
		 \end{cases}
	\]
	We apply Proposition \ref{prop:2} to $N \times [0,1] \subset B \times [0,1]$, the latter being equipped with the metric $\tilde{g} := g(t) + dt^2$, obtaining a 
	constant $r^{**}$ such that for $r \in (0, r^{**})$ the  $n$-dimensional distance tube $\tilde{T}(r) \subset B \times [0,1]$ with the induced metric satisfies $C$. 
	Now, for $r$ small enough we can find $\epsilon > 0$ such that
	\[
		\left(\tilde{T}(r) \cap \left(B \times [0, \epsilon)\right), \rline{\tilde{g}}{\tilde{T}(r) \cap \left(B \times [0, \epsilon)\right)} 
			\right) = \left( T(r) \times [0, \epsilon), g_D \right)
	\]
	and
	\[
		\left(\tilde{T}(r) \cap \left(B \times (1 - \epsilon, 1]\right), \rline{\tilde{g}}{\tilde{T}(r) \cap \left(B \times (1 - \epsilon, 1]\right)} \right)
			= \left( T(r) \times (1- \epsilon, 1], \rline{h}{T(r)} + dt^2 \right).
	\]
	By extending the metric constantly for $t \geq 1$ and relabelling, the claim follows.
\end{proof}

Therefore, if we choose as the final radius in the bending process in Lemma \ref{lem:step2} a radius $r$ that fulfills $0 < r < \underline{r} := \min\{r^*, r^{**} \}$, we can simply replace the metric on $D \cap \left(M \times [t(\overline{s}), \infty)\right) = T(r) \times [t(\overline{s}), \infty)$ by the one constructed in Lemma \ref{lem:step3}. This completes the proof of Theorem \ref{thm:12}.

\section{Vertical rescaling of Riemannian submersions}
\label{sec:subm}
%
%

In this section, we consider Riemannian submersions $\pi : (M^n, g_M) \rightarrow (B^{n-k}, g_B)$ of closed Riemannian manifolds and show that the total space $M^n$ with a vertically rescaled metric $g_M^t$ satisfies a given curvature condition, provided the fibers do so in an appropriate way. To define $g_M^t$, recall that the fiber $F(b) := \pi^{-1}\left(\{b\}\right)$ over a point $b \in B$ is naturally an embedded submanifold of $M$ and will henceforth be endowed with the induced metric, denoted by $g_{F(b)}$. This implies the existence of a smooth orthogonal splitting of the tangent bundle $TM$, given pointwise by
\[
	T^v_pM := T_pF\left(\pi(p)\right), \qquad T_p^hM := {T_pF\left(\pi(p)\right)}^{\bot}.
\]
The corresponding orthogonal projections are smooth maps and will be denoted by $w \mapsto \vertproj{w}$ and $w \mapsto \horzproj{w}$, respectively. 
\par
Now, by shrinking the metric $g_M$ in the direction of the fibers we get a new metric $g_M^t$, $t > 0$, defined by
\begin{align*}
	g_M^t (w_1, w_2) &:= t^2 g_M (\vertproj{w_1}, \vertproj{w_2}) + g_M(\horzproj{w_1}, 
												\horzproj{w_2}) \\
									 &=t^2 g_{F(\pi(p))}(\vertproj{w_1}, \vertproj{w_2}) +
									 				\pi^*g_B (w_1, w_2),
\end{align*}
for $w_1, w_2 \in T_pM$, $p \in M$. With respect to this deformed metric the map $\pi : (M^n, g_M^t) \to (B^{n-k}, g_B)$ continues to be a Riemannian submersion with the same decomposition $TM = T^vM \oplus T^hM$ of the tangent bundle into a vertical and  horizontal subbundle.
\begin{theorem}
	\label{thm:6}
	Let $C \subset \algcurvop{\reals^n}$ be a curvature condition. Let $\pi : (M^n, g_M) \rightarrow (B^{n-k}, g_B)$
	be a Riemannian submersion, $M^n$ and $B^{n-k}$ being closed manifolds. If $C$ satisfies an inner cone condition with respect to any curvature operator corresponding 
	to
	\[
		R_{\left(F(b) \times \reals^{n-k}, g_{F(b)} + g_{\reals^{n-k}}\right)}(p),
	\]
	with $b \in B$, $p \in F(b)$, then there is $t_* > 0$ such that for each $t \in (0, 
	t_*)$ the Riemannian manifold $\left(M, g_M^t\right)$ satisfies $C$.
\end{theorem}
\begin{remark}
	If $C$ happens to be a convex cone, the given condition simplifies to $\tilde{R}_{F(b) \times \reals^{n-k}}(p) \in C$.
\end{remark}
For the \emph{proof} of Theorem \ref{thm:6}, recall that the behavior of a Riemannian submersion is determined by two tensorial invariants of type $(2,1)$ which are given by
\begin{align*}
	T_XY &= \horzproj{\left(\nabla_{\vertproj{X}}\vertproj{Y}\right)} + \vertproj{\left(\nabla_{\vertproj{X}}\horzproj{Y}\right)},
					\\
	A_XY &= \horzproj{\left(\nabla_{\horzproj{X}}\vertproj{Y}\right)} + \vertproj{\left(\nabla_{\horzproj{X}}\horzproj{Y}\right)},
\end{align*}
where $X,Y \in \crosssection{TM}$. The tensor $T$ essentially describes the second fundamental form of the fibers, whereas the tensor $A$ serves as an obstruction to the integrability of the horizontal distribution $T^hM$.
\par
Now, using the well-known O'Neill formulas and the variational behavior of $A$ and $T$ with respect to vertically rescaled metrics (see \cite{Besse}, Theorem 9.28, and Lemma 9.69, respectively) we derive the following set of
equations. 

\begin{lemma}
	\label{lem:1}
	Let $v_1, \ldots, v_4 \in T_p^vM$ and $h_1, \ldots, h_4 \in T_p^hM$, $b = \pi(p)$. Then
	the $(4,0)$-curvature tensor of $(M^n, g_M^t)$ is given by
	\begin{align*}
		R_M^t (v_1, v_2, v_3, v_4) &= t^2 R_{F(b)}(v_1, v_2, v_3, v_4) \\
				&\ghost{=} - t^4 \bigl( g_M\left(T_{v_2}v_3, T_{v_1}v_4 \right) - g_M\left(T_{v_1}v_3, T_{v_2}v_4 \right) \bigr) \\
		R_M^t(v_1, v_2, v_3, h_1) &= t^2 R_M(v_1, v_2, v_3, h_1) \\
				&\ghost{=} - (t^2 - t^4) \bigl( g_M\left(T_{v_1}v_3, A_{h_1}v_2 \right) + g_M\left(T_{v_2}v_3, A_{h_1}v_1 \right) 
					\bigr)\\
		R_M^t(v_1, v_2, h_1, h_2) &= t^2 R_M(v_1, v_2, h_1, h_2) \\
				&\ghost{=} + (t^2 - t^4) \bigl( g_M\left(A_{h_1}v_2, A_{h_2}v_1 \right) - g_M\left(A_{h_1}v_1, A_{h_2}v_2 \right) 
					\bigr) \\
		R_M^t(h_1, v_1, h_2, v_2) &= t^2 R_M(h_1, v_1, h_2, v_2) + (t^2 - t^4) g_M \left(A_{h_1}v_2, A_{h_2}v_1 \right) \\
		R_M^t(h_1, h_2, h_3, v_1) &= t^2 R_M(h_1, h_2, h_3, v_1) \\
		R_M^t(h_1, h_2, h_3, h_4) &= t^2 R_M (h_1, h_2, h_3, h_4) + (1 - t^2) (\pi^*R_B)(h_1, h_2, h_3, h_4).
	\end{align*}
\end{lemma}
\begin{proof}[Proof of Theorem \ref{thm:6}]
	Let $U \subset M$ be an open subset such that there exists an $g_M$-orthonormal frame $V(q) + H(q)$, with $V(q) = (v_1(q), \ldots, v_k(q))$, $H(q) = (h_1(q), \ldots, 
	\allowbreak h_{n-k}(q))$, $v_i, h_j \in \crosssection{U, TM}$ with
	$v_i(q) \in T_q^vM$ and $h_j(q) \in T_q^hM$, $q \in U$. By setting $V^t(q) := (v_1^t(q), \ldots, v_k^t(q))$, $v_j^t := \frac{1}{t} v_j$, we obtain a corresponding 
	frame on $U$ with respect to the metric $g_M^t$. 
	\par
	Let $\tilde{R}^t_M(q)$ denote the $(V^t(q)+H(q))$-pullback of $R_{(M, g^t_M)}(q)$ as in Definition \ref{def:007}, and 
	denote by $\tilde{R}^t_F(q)$ the $\left(V^t(q) + (e_{k+1},\ldots, e_n)\right)$-pullback of the curvature tensor corresponding to the manifold $\left(F(\pi(q)) \times \reals^{n-k}, 
	t^2g_{F(\pi(q))} + g_{\reals^{n-k}}\right)$.
	\par
	Now, Lemma \ref{lem:1} shows that $\tilde{R}_M^t$ can be decomposed as
	\begin{equation}
		\label{eq:3}
		\tilde{R}_M^t = \tilde{R}^t_{F} + E^t,
	\end{equation}
	with $E^t : U \rightarrow \algcurvop{\reals^n}$ collecting the remaining terms. By assumption we can find an $\Or(n)$-invariant cone 
	$\tilde{C}$ as in the proof of Theorem \ref{thm:1} such that 
	$\tilde{C} \backslash K \subset C$ and that for some $p \in U$ and $\epsilon > 0$ we have $\ball{\epsilon}{\tilde{R}_F^1(p)} \subset \tilde{C}$. 
	$\tilde{C}$ being a cone, we get $\ball{\frac{\epsilon}{t^2}}{\tilde{R}^{t}_{F}(p)} \subset \tilde{C}$ for $t > 0$ and then 
	$\ball{\frac{\epsilon}{t_p^2}}{\tilde{R}^{t_p}_{F}(p)} \subset \tilde{C} \backslash K$ for some $t_p > 0$. By shrinking $U$, if necessary, we find $\delta > 0$ 
	such that 
	\[
		\ball{\frac{\delta}{t^2}}{\tilde{R}_F^t(q)} \subset C
	\] 
	for all $q \in U$ and $t \in (0, t_p)$. Thus, it is enough to show the existence of some $t_* \in (0, t_p]$ with 
	\begin{equation}
		\label{eq:4}
		\norm{E^t(q)} < \frac{\delta}{t^2}
	\end{equation}
	for $q \in U$ and $t \in (0, t_*)$. Because $M$ can be covered by finitely many neighborhoods $U$, this will finish the proof.
	\par
	Now, to prove \eqref{eq:4}, a case by case study of the equations of 
	Lemma \ref{lem:1} yields, with $i,j,m,l \in \{1,\ldots, k\}$, $r,s,u,v \in \{1, \ldots, n-k \}$,
	\begin{align*}
		E^t (e_i, e_j, e_m, e_l) &= \left(\tilde{R}^t_M - \tilde{R}^t_F \right)(e_i, e_j, e_m, e_l) \\
			&= \left( t^{-4}R^t_M - t^{-2}R_{F(\pi(p))} \right)(v_i, v_j, v_m, v_l) \\
			&= g_M\left(T_{v_j}v_m, T_{v_i}v_l \right) - g_M\left(T_{v_i}v_m, T_{v_j}v_l \right).\\
		E^t (e_i, e_j, e_m, e_{k+r}) &= t^{-1} R_M(v_i, v_j, v_m, h_r) \\
			&\ghost{=} - \left(t^{-1} - t \right) \left( g_M\left(T_{v_i}v_m, A_{h_r}v_j \right) + g_M\left(T_{v_j}v_m, 
					A_{h_r}v_i \right) \right)\\
		E^t (e_i, e_j, e_{k+r}, e_{k+s}) &= R_M(v_i, v_j, h_r, h_s) \\
			&\ghost{=} + \left(1 - t^2 \right) \left( g_M\left(A_{h_r}v_j, A_{h_s}v_i \right) - g_M\left(A_{h_r}v_i, A_{h_s}v_j 
				\right) \right) \\
		E^t(e_{k+r}, e_i, e_{k+s}, e_j) &= R_M(h_r, v_i, h_s, v_j) + (1 - t^2) g_M \left(A_{h_r}v_j, A_{h_s}v_j \right) \\
		E^t(e_{k+r}, e_{k+s}, e_{k+u}, e_i) &= t R_M(h_r, h_s, h_u, v_i) \\
		E^t(e_{k+r}, e_{k+s}, e_{k+u}, e_{k+v}) &= t^2 R_M (h_r, h_s, h_u, h_v) + (1 - t^2) (\pi^*R_B)(h_r, h_s, h_u, h_v).
	\end{align*}
	Thus, we can find a constant $C > 0$ such that $\norm{E^t(q)} \leq \frac{C}{t}$ for all $q \in U$, which evidently implies \eqref{eq:4} for some $t_*$.
\end{proof}

\section{Proof of Theorem C}
\label{sec:class}

	We consider first the case that $M$ is spin. Stolz proved that the vanishing of the $\alpha$-invariant implies that $M$ is spin cobordant to 
	the total space $N^n$ of a fiber bundle with fiber $\mathbb{H}P^2$ (cf. Theorem B of \cite{Stolz}). Now, $\mathbb{H}P^2 \times \reals^{n-8}$ satisfies 
	$C$ by assumption, hence - using Theorem \ref{thm:6} - $N$ can be equipped with a metric satisfying $C$. Of course, $M$ might be nullcobordant, in which 
	case it is obviously coborbant to the $n$-dimensional sphere $S^n$. By a theorem in \cite{GromovLawson1} $M$ can be obtained from $N$ (or $S^n$) by surgeries of 
	codimension at least $3$. This shows the claim in the case of $M$ being spin. \par
	In the non-spin case, it was also observed by Gromov and Lawson that two simply connected manifolds which are oriented cobordant can in fact be obtained from one 
	another by surgeries of codimension at least $3$. Therefore it suffices to give a list of generators of the oriented cobordism ring $\Omega^{SO}_*$, all of which 
	carry a metric satisfying $C$, if their dimension matches $n$. In the following, we will see that - again - the list proposed by Gromov and Lawson for the 
	case of positive scalar curvature suffices for our purposes (see \cite{GromovLawson1} for more details).\par
	The ring $\Omega^{SO}_*$ modulo torsion is generated by complex projective spaces $\mathbb{C}P^k$ and Milnor manifolds $H_{k,m}$ given as hypersurfaces of degree 
	$(1,1)$ in $\mathbb{C}P^k \times \mathbb{C}P^m$, $k \leq m$.  The former obviously carry a metric satisfying $C$, and so do the latter, by application of 
	Theorem \ref{thm:6} as above. In order to see this, recall that $H_{k,m}$ can be defined as
	\[
		H_{k,m} := \set{ \left([w_0, \ldots, w_k], [z_0, \ldots, z_m]\right) \in \mathbb{C}P^k \times \mathbb{C}P^m}
			{ \sum_{j=1}^k w_j z_j = 0}.
	\]
	Together with the projection $H \to \mathbb{C}P^k$ onto the first factor and using the induced metric, $H$ can be easily 
	endowed with a metric turning the projection into a Riemannian submersion, with fibers being isometric to $\mathbb{C}P^{m-1}$.
	\par
	Generators of the torsion of $\Omega^{SO}_*$ consist of two types of manifolds. The first type is a so-called Dold manifold $D_{k,m}$ defined by 
	\[
		D_{k,m} := \left(S^k \times \mathbb{C}P^m\right)/ \mathbb{Z}_2,
	\]
	where the $\mathbb{Z}_2$-action is given by $(p, [z]) \mapsto (-p, [\overline{z}])$. The obvious metric on this manifold is non-flat and has non-negative curvature 
	operator, hence it satisfies $C$, possibly after rescaling. \\
	The second type can be constructed as follows. Define $P_{k,m} := (D_{k,m} \times S^1) / \mathbb{Z}_2$ with $\mathbb{Z}_2$-action the map $([p, [z]], \phi) \mapsto 
	([r(p), [z]], -\phi)$, $r : S^k \to S^k$ being a reflection about a hyperplane. With the help of the induced projection $\psi_{k,m} : P_{k,m} \to S^1$, the second 
	type of torsion generators is constructed as
	\[
		V := \set{ (x_1, \ldots, x_l) \in P_{k_1,m_1} \times \ldots \times P_{k_l,m_l}}{\psi_{k_1,m_1}(x_1)\cdot \ldots \cdot 
				\psi_{k_l,m_l}(x_l) = 1}.
	\]
	The map $(x_1, , \ldots, x_l) \mapsto \left(\psi_1(x_1), \ldots, \psi_l(x_l)\right)$, where $\psi_j = \psi_{k_j,m_j}$, defines a submersion $\pi : V \to T^{l-1} := \set{ (t_1, \ldots, t_l) \in \left(S^1\right)^l}{ t_1\cdot \ldots \cdot t_l = 1}$. Since $P_{k,m}$ is locally isometric to $D_{k,m} \times \reals$, $V$ is locally isometric to a product of $l$ Dold manifolds and $\reals^{l-1}$ and thus satisfies $C$, again possibly after rescaling.
This completes the proof of Theorem C.
\vspace{3mm}
\par
In order to deduce Corollary D, 
we have to consider the condition 
\[
	C_{\epsilon} := \set{R \in \algcurvop{\reals^n}}{ R > -\epsilon \norm{R}},
\]
for a given $\epsilon > 0$. $C_{\epsilon}$ is not a convex condition for small $\epsilon > 0$, nevertheless we have

\begin{proposition}
	The curvature condition $C_{\epsilon}$ satisfies an inner cone condition with respect to any $0 \neq S \in \algcurvop{\reals^n}$ with nonnegative eigenvalues. In 
	particular, $C_{\epsilon}$ is stable under surgeries of codimension at least 3.
\end{proposition}

\begin{proof}
	We show that $C_{\epsilon}$ satisfies an inner cone condition with respect to $S$, where we assume w.l.o.g. $\norm{S} = 1$. Fix $R \in C_{\epsilon}$, then $R 
	+ \epsilon' \norm{R} > 0$ for some $0 < \epsilon' < \epsilon$. We will establish the existence of a $\delta = \delta(R)$ such that $\ball{t \delta}{R + tS} \subset 
	C_{\epsilon}$ for all $t \geq 0$. Equivalently, we claim that for $T \in \algcurvop{\reals^n}$ with $\norm{T} < \delta$ and $\omega \in \bigwedge^2\reals^n$ 
	with $\norm{\omega} = 1$ the function
	\[
		f(t) := \innerprod{(R + t(S + T))(\omega)}{\omega} + \epsilon \norm{(R + t(S + T))} 
	\]
	is positive. Now, for $0 \leq t \leq t_0 := (1 + \epsilon) \norm{R}$ we estimate
	\begin{align*}
		f(t) &\geq \innerprod{R (\omega)}{\omega} - t \delta + \epsilon \left( \norm{R + tS} - t \delta \right) \\
			&\geq \innerprod{R (\omega)}{\omega} + \left( \epsilon - \delta (1 + \epsilon)^2 \right) \norm{R} \\
			&\geq \innerprod{R (\omega)}{\omega} + \epsilon' \norm{R} \\
			&> 0,
	\end{align*}
	provided $\delta < \frac{\epsilon -  \epsilon'}{(1+\epsilon)^2}$; we additionally used the inequality $\norm{R + tS} \geq \norm{R}$.\\
	Similary, in order to arrive at an analogous estimate for $t \geq t_0$, we first choose $\omega_0 \in \bigwedge^2 \reals^n$, $\norm{\omega_0} = 1$, with 
	$\innerprod{S(\omega_0)}{\omega_0} = 1$. Then
	\[
		\norm{R + tS + tT} \geq \innerprod{R (\omega_0)}{\omega_0} + (1- \delta) t \geq -\epsilon \norm{R} + (1-\delta)t,
	\]
	and hence 
	\[
		f(t) \geq \innerprod{R (\omega)}{\omega} - \delta t + \epsilon \left( (1-\delta)t -\epsilon \norm{R} \right) 
			= \innerprod{R (\omega)}{\omega} - \epsilon^2 \norm{R} + \left( \epsilon - \delta (1+ \epsilon) \right) t.
	\]
	For $\delta < \frac{\epsilon}{1 + \epsilon}$ the coefficient in front of $t$ is positive, thus we obtain using $t \geq t_0 = (1 + \epsilon )\norm{R}$
	\[
		f(t) \geq \innerprod{R (\omega)}{\omega} + \left( \epsilon - \delta (1 + \epsilon)^2 \right) \norm{R} 
			\geq \innerprod{R (\omega)}{\omega} + \epsilon' \norm{R} 
			> 0,
	\]
	as above. This finishes the proof.
\end{proof}

Putting these things together, Corollary D follows.

\section{Equvariant Surgery}
\label{sec:equivariant}

There exist corresponding equivariant versions of the above surgery and gluing theorems A and B.

\begin{theorem}
	\label{thm:10}
	Let $C$, $N_i^l \subset M_i^n$ and $\phi$ be given as in Theorem B. 
	Additionally, suppose $G$ is a compact Lie group 
	acting isometrically on $M_i^n$, $i=1,2$, so that $\gamma(N_i) = N_i$ for all $\gamma \in G$ and such that $\phi$ is 
	$G$-equivariant with respect to the induced actions of $G$ on $\nu N_i$.
	\par
	Then, $M_1 \#_{\phi} M_2$, the joining of $M_1$ with $M_2$ along $\phi$, carries a metric satisfying $C$ on which $G$ acts 
	isometrically.
\end{theorem}

$G$-equivariant surgery is a special case of the joining of two $G$-manifold via a $G$-equivariant vector bundle isomorphism. More precisely, we have

\begin{definition}
	Suppose there are a closed subgroup $H \subset G$ and orthogonal representations of $H$ on $\reals^c$ and $\reals^{d+1}$, which induce actions on $S^d \subset 
	\reals^{d+1}$ and $S^{d+c} \subset \reals^{d+1} \times \reals^{c}$. Let $N$ be a submanifold of an $n$-dimensional manifold $M^n$, $N$ being $G$-equivariantly 
	diffeomorphic to $G \times_H S^{d}$.
	If there is a $G$-equivariant isomorphism of the normal bundles of $N$ and $G \times_H 	S^{d} \subset G \times_H 	S^{d+c}$, then \emph{$G$-equivariant surgery of 
	dimension $d$ and codimension $c$}, denoted $\chi_G(M, N)$, is given by joining of $G \times_H S^{d+c}$ and $M$ along this isomorphism.
\end{definition}

Since the normal bundle of $G \times_H S^{d}$ in $G \times_H S^{d+c}$ is equivariantly diffeomorphic to $G \times_H \left(S^{d} \times \reals^c \right)$, the submanifold $N \subset M$ as above is required to admit a tubular neighborhood which is equivariantly diffeomorphic to $G \times_H \left(S^{d} \times D^c\right)$ for some open ball $D^c \subset \reals^c$. This leads to a somewhat more common (but equivalent) way of stating equivariant surgery: After removing a region like $G \times_H \left(S^{d} \times D^c\right)$, a new region like $G \times_H \left(\bar{D}^{d+1} \times S^{c-1}\right)$ is pasted in along the common boundary:
\[
	\chi_G(M, N) = \left[M \backslash G \times_H \left(S^{d} \times D^c\right)\right] \cup_{G \times_H \left(S^{d} \times S^{c-1}\right)}
		\left[G \times_H \left(\bar{D}^{d+1} \times S^{c-1}\right)\right].
\]

Since $G \times_H S^{c+d}$ is naturally a fiber bundle over $G/H$  with fibers diffeomorphic to $S^{c+d}$, the construction outlined in Theorem \ref{thm:6} yields a $G$-invariant metric such that $C$ as above is satisfied. We therefore deduce from Theorem \ref{thm:10} the following

\begin{theorem}
	\label{thm:11}
	Let $C \subset \algcurvop{\reals^n}$ be a curvature condition satisfying an inner cone condition with respect to 
	$R_{S^{c-1} \times \reals^{n-c+1}}$, $c \in \{3,\ldots,n\}$. Let $(M^n, g_M)$ be a Riemannian manifold satisfying $C$ and suppose there is a 
	compact Lie group $G$ acting isometrically on $M$.
	\par Then a manifold 
	obtained from $M^n$ by performing $G$-equivariant surgery of codimension at least $c$ also admits a metric satisfying $C$.
\end{theorem}

\begin{remark}
	The validity of Theorem \ref{thm:11} in the case of positive scalar curvature was first observed in \cite{Bergery} (cf. also \cite{Hanke}).
\end{remark}

The \emph{proof} of Theorem \ref{thm:10} consists of an application of the following equivariant analogue of Theorem \ref{thm:12} which is  
formulated in 

\begin{theorem}
	\label{thm:13}
	Additionally to the assumptions of Theorem \ref{thm:12}, let $G$ be a compact Lie group which acts isometrically on $(M, g_M)$,
	leaving $N$ invariant. Furthermore, suppose that $G$ acts isometrically on $(N, g_N)$ as well as on $(\nu N, g_{\nu N})$ (by 
	the induced action) and that the connection $\nabla$ commutes with this action.
	\par
	Then, for $\overline{r} > 0$ there is $\underline{r} \in (0, \overline{r})$ such that for every $r \in (0, \underline{r})$ there exists a complete metric $g_D$ on the open 
	manifold $D := M \backslash N$ with the following properties:
	\begin{enumerate}
		\item $g_D$ satisfies $C$.
		\item $g_D$ coincides with $g_M$ on $M \backslash D(\overline{r})$, where $D(\overline{r}) = \set{x \in M}{d_{g_M}(x, N) < \overline{r}}$.
		\item In a neighborhood $U \subset M$ of $N$ the region $(U \backslash N, g_D)$ is isometric to 
			\[
				\left(\nu^rN \times (0, \infty), \rline{h}{\nu^rN} + dt^2 \right),
			\]
			where $h$ is the connection metric determined by $g_N$, $g_{\nu N}$ and $\nabla$.
		\item $G$ acts isometrically on $(D, g_D)$. Moreover, for any $t_0 \in (0, \infty)$, the submanifold $\nu^rN \times 
				\{t_0\} \subset (U \backslash N, g_D)$ is invariant under the $G$-action.
	\end{enumerate}
\end{theorem}

The proof of Theorem \ref{thm:12} given in section \ref{sec:mainproof} carries over to the equivariant case almost verbatim, due to the following observation.

\begin{proposition}
	Let $(D, g_D) \subset (M \times \reals, g_M + dt^2)$ be given by \eqref{eq:38} with $\gamma(\nu, s)$ being constructed as in the first two steps of the proof of 
	Theorem 	\ref{thm:12} in section \ref{sec:mainproof}. Suppose $f : (M, g_M) \to (M, g_M)$ is an isometry with $f(N) = N$.	Then, there is a unique isometric 
	extension $\tilde{f} : (D, g_D) \to (D, g_D)$ coinciding with $f$ on $M \backslash D(\overline{r})$. 
\end{proposition}

\begin{proof}
	Consider $\bar{f} :M \times \reals \to M \times \reals$, $(p,t) \mapsto (f(p), t)$. Since $f(N) = N$, $f$ maps geodesics which start orthogonal to $N$ to 
	geodesics of the same type, thus we have $\bar{f}(\gamma(\nu, s)) = \gamma(df(\nu), s)$. This immediately implies $\bar{f}(D) = 
	D$ by construction of $D$. Since $\bar{f}$ is an isometry, so is its restriction $\tilde{f} := \rline{\bar{f}}{D}$.
\end{proof}

Thus the first two construction steps can be reused, only the blending step needs to be adjusted in order to ensure that the induced action of $G$ be isometric. This is accomplished by averaging the family $g(t)$ of metrics in step $3$ which hence is being replaced by
\[
	\tilde{g}(t) = \int_G \gamma^*g(t) \, dm(\gamma).
\]
This averaged metric $\tilde{g}(t)$ coincides with $g(t)$ where it is already $G$-invariant, namely for $t \in [0, \epsilon]$ as well as for $t \in [1- \epsilon, 1]$. The former holds because the metric equals the induced metric of a distance tube around $N$, the latter holds since the connection metric $h$ is made up of equivariant pieces. Indeed, if $g \in G$, then $g$ acts on $\nu N$ via the differential $dg : \nu N \to \nu N$. Then the differential of $dg$ maps the horizontal distribution $\horzdistr_{\nu} \subset T_{\nu}\nu N$, which is determined by the connection, onto $\horzdistr_{dg(\nu)}$, since $\nabla$ commutes with $G$. Now, $dg$ preserves both the fiber metric on the horziontal distribution as well as the one given on the vertical distribution, which implies the preservation of $h$ under the action of $G$.\par
This completes the proof of Theorem \ref{thm:13}.

\section{Surgery of conformally flat manifolds}
\label{sec:conformallyflat}

A similar surgery result holds for the class of conformally flat manifolds, but - in contrast to the other cases -, here only $0$-surgery, i.e. connected sum constructions, can be expected, since $S^{n-l} \times \reals^{l}$, $l > 0$, $n \geq 3$, is conformally flat precisely for $l = 1$.

\begin{theorem}
	\label{thm:4}
	Let $C \subset \algcurvop{\reals^n}$ be a curvature condition satisfying an inner cone condition with respect to 
	$R_{S^{n-1} \times \reals}$. Assume that $C$ is star-shaped with respect to $0$. Suppose $(M^n_i, g_{M_i})$, $i=1,2$, are $n$-dimensional conformally flat 
	Riemannian manifolds satisfying $C$. Then the connected sum $M_1 \# M_2$ also admits a conformally flat metric satisfying $C$.
\end{theorem}

\begin{remark}
	The validity of Theorem \ref{thm:4} in the case of positive scalar curvature was noticed in \cite{SchoenYau1}.
\end{remark}

Conceivably, Theorem \ref{thm:4} follows directly from

\begin{theorem}
	\label{thm:8}
	Let $C \subset \algcurvop{\reals^n}$ be a curvature condition satisfying an inner cone condition with respect to $R_{S^{n-1} 
	\times \reals}$ and suppose $C$ is star-shaped with respect to the origin. Let $(M^n, g)$ be a conformally flat manifold 
	satisfying $C$. Given $p \in M$ and any 	$\epsilon > 0$ there exists $\delta > 0$  such that for any $\gamma \in (0, \delta)$ there exists a positive 
	function $\sigma \in C^{\infty}(M \backslash \{p\})$ which defines a complete 
	metric $g_D := \sigma^2 g$ on the open manifold $D := M \backslash \{ p \}$ with the following properties:
	\begin{enumerate}
		\item \label{enum:thm:8-1} $(D, g_D)$ is conformally flat and satisfies $C$.
		\item \label{enum:thm:8-2} $g_D$ coincides with $g$ on $M \backslash D(\epsilon)$, where $D(\epsilon) = \set{x \in M}{d^{g}(p,x) < \epsilon}$.
		\item \label{enum:thm:8-3} There exists a neighborhood $U \subset D(\epsilon)$ of $p$ such that
			$\left( U \backslash \{ p \}, g_D \right)$ is isometric to
			$\left(S^{n-1}(\gamma) \times (0, \infty)\right)$ with $S^{n-1}(\gamma)$ being equipped with the standard metric of constant curvature $\gamma^{-2}$.
	\end{enumerate}
\end{theorem}

The \emph{proof} of this theorem is based on an extension of a method developed in \cite{MicallefWang}. Before going into the details of the construction, it is useful to collect some formulas which show the influence of a conformal deformation on the curvature of a Riemannian manifold.

\begin{lemma}
	\label{lem:4}
	Let $\sigma$ denote a positive, smooth function on a Riemannian manifold $(M^n, g)$. 
	\begin{enumerate}
		\item The gradient of a function $f \in C^{\infty}(M)$ with respect to $\sigma^2 g$ is given by
			\[
				\nabla^{\sigma^2g}f = \sigma^{-2}\nabla^gf.
			\]
		\item The Levi-Civita connections are related by
			\[
				\nabla^{\sigma^2g}_XY = \nabla^{g}_XY + \sigma^{-1} \left( d\sigma(X)Y + d\sigma(Y)X - g(X,Y) \nabla^g\sigma \right)
			\] 
		\item \label{lem:4-3} The $(2,0)$-hessian of a function $f \in C^{\infty}(M)$ is given by 
			\[
				\Hess^{\sigma^2g}(f) = \Hess^{g}(f) + \sigma^{-1}\left(d\sigma (\nabla^gf)\, g - d\sigma \odot df \right),
			\]
			where $d\sigma \odot df := d\sigma \otimes df + df \otimes d\sigma$.
		\item \label{lem:4-4} In case $\sigma = w \circ f$, where $w \in C^{\infty}(\reals)$ and $f \in C^{\infty}(M)$, the $(4,0)$-curvature tensor 
					is given by
			\begin{align*}
				R_{(M, \sigma^2 g)} &= \sigma^2 \Biggl[ R_{(M,g)} - \left(\frac{w'}{w}\right) \circ f \, g \wedge \left( 2 \Hess^{g}(f) +  
					\left(\frac{w'}{w}\right) \circ f \sqnorm{df}_{g}	g \right) \\
				&\ghost{= \sigma^2 \Biggl[ } - 2\left( \frac{w''}{w} - 2  \left(\frac{w'}{w}\right)^2 \right) \circ f \,(g \wedge df^2 )\Biggr].
			\end{align*}
	\end{enumerate}
\end{lemma}

Now we turn to the construction of the function $\sigma$ as described in Theorem \ref{thm:8}. 
First of all, due to the conformal flatness we are able to choose a function $v \in C^{\infty}(V)$ in a neighborhood $V \subset \ball{\epsilon}{p}$ of $p$ such that $v^2 g_M$ is flat. By possibly rescaling $v$ and shrinking $V$, we can suppose that $v(p) = 1$, $2^{-1} \leq v \leq 2$, and that $V$ is a diffeomorphic image of some open ball $\ball{\epsilon'}{0} = \set{\nu\in T_pM}{g(\nu, \nu) < \epsilon'}$ under the mapping $\exp := \exp^{v^2g}_p$. Let $r(x) := d^{v^2g}(p,x)$ denote the distance function to $p$ with respect to this flat metric. Let $\phi : \reals \to [0,1]$ be a smooth cut-off function satisfying $\phi|_{[0, \frac{1}{2}]} \equiv 1$ and $\phi|_{[1, \infty)} \equiv 0$. Given $\lambda > 0$, which will be determined in due course, we define
\[
	\phi_{\lambda}(x) := \phi\left(\frac{r(x)}{\lambda}\right)
\]
on $V$. Using this rescaled cut-off function we set 
\[
	v_{\lambda} := \sqrt{\phi_{\lambda} v^2 + \left(1- \phi_{\lambda}\right)} = \sqrt{ 1 + \phi_{\lambda} (v^2 -1 )}
\]
Then the estimate $2^{-1} \leq v_{\lambda} \leq 2$ continues to hold. We set $q := \frac{v_{\lambda}}{v}$, and get $4^{-1} \leq q \leq 4$.\\
Furthermore, as in \cite{MicallefWang}, we consider the differential equation
\begin{equation}
	\label{eq:23}
	\frac{u'(r)}{u(r)} = - \frac{\alpha(r)}{r},
\end{equation}
with $\alpha : [0, \infty) \to [0,1]$ yet to be determined such that a solution $u : [0, \infty) \to [1, \infty)$ can be constructed matching the following conditions:
\begin{enumerate}
	\item $u \equiv 1$ on $[r_0, \infty)$ for some $r_0 >0$.
	\item $u(r) = \frac{\gamma}{r}$ for all $r \leq \lambda$, where $\gamma$ is given as in the statement of Theorem \ref{thm:8}.
	\item $(D, g_D) := (M \backslash \{p\}, \sigma^2 g)$, satisfies $C$. Here, $\sigma \in C^{\infty}(M \backslash 
					\{p\})$ is given by $\sigma = u v_{\lambda}$ where both $u$ and $v_{\lambda}$ are defined and by $1$ otherwise, 
					which yields a smooth function by construction.
\end{enumerate}
\par
The purposes served by the two complementary deformations is to first address (via $u$) the problem of deforming the distance spheres of varying radii $r$ with respect to the flat metric $v^2g$ around $p \in M$ towards a uniform radius $\gamma$ (still with respect to the metric $v^2g$), which will make the new end of the manifold roughly look like a cylinder, and, secondly, to blend (using $v_{\lambda}$) the original metric smoothly into the aforementioned flat metric in order to make the end look exactly like a cylinder.
\par
Once $u$ and $v_{\lambda}$ are determined, Theorem \ref{thm:8} will be proven. Indeed, \ref{enum:thm:8-1}  and \ref{enum:thm:8-2} of Theorem \ref{thm:8} hold true obviously; for the validity of \ref{enum:thm:8-3} observe that the metric $g_D$ on $\{ r < \frac{\lambda}{2} \}$ has the form
\begin{align*}
	\sigma^2 g = \frac{\gamma^2}{r^2} v^2 g = \frac{\gamma^2}{r^2} \left( dr^2 + r^2 g_{S^{n-1}} \right) = ds^2 + g_{S^{n-1}(\gamma)},
\end{align*}
where $s := \gamma \log(r) \in \left(-\infty, \gamma \log\left(\frac{\lambda}{2}\right)\right)$. This also implies the completeness of $g_D$.
\par
Again, before the various parameters in the bending process can be specified, a close inspection and decomposition of the involved curvature tensor is necessary. To this end, we choose for any $\nu \in T^1_pM = \set{ \nu \in T_pM}{ g(\nu, \nu) = 1}$ and $r > 0$ an orthonormal basis $B(\nu, r) = (b_1, \ldots, b_n)$, $b_i \in T_{\exp(r\nu)}M$, with $b_n = d\exp(r\nu)(\nu) = \nabla^{v^2g}r$. Using the notation $w \cdot B(\nu, r) := \left(w(\exp(r\nu)) b_1, \ldots, w(\exp(r\nu)) b_n\right)$ for a function $w$, we denote using the notation of Definition \ref{def:007}
\begin{gather*}
	\text{the }(uq)^{-1} \cdot B(\nu, r)\text{-pullback of }R_{(D, u^2v_{\lambda}^2g)}\text{ by }\tilde{R}_D(\nu, r),\\
	\text{the }v \cdot B(\nu, r)\text{-pullback of }R_{(M,g)}\text{ by }\tilde{R}_M (\nu, r)\text{ and }\\
	\text{the }q^{-1} \cdot B(\nu, r)\text{-pullback of }R_{(M, v_{\lambda}^2g)}\text{ by }\tilde{R}^{\lambda}_M (\nu, r).
\end{gather*}
\par Now we can state
\begin{proposition}
	Given the identifications made above, the identity
	\label{prop:10}
	\begin{equation}
		\label{eq:28}
		\begin{split}
			\tilde{R}_D (\nu, r) = \, &u^{-2} \tilde{R}^{\lambda}_M (\nu, r) + (uq)^{-2} 
					\biggl[\frac{\alpha(r)\left(2-\alpha(r)\right)}{r^2} R_{S^{n-1} \times \reals} + \\ 
					&\frac{2\alpha'(r)}{r} g_{\reals^n} \wedge 
						\left(e_n^{\flat}\otimes e_n^{\flat}\right) + \frac{2\alpha(r)}{r}E^{\lambda}(\nu, r)\biggr]
		\end{split}
	\end{equation}
	holds for $(\nu, r) \in T_p^1M \times (0, \epsilon')$, where $E^{\lambda} : T_p^1M \times (0, \epsilon') \to \algcurvop{\reals^n}$ has a bounded 
	image, i.e.
	\[
		\norm{E^{\lambda}(\nu, r)} \leq C_1
	\]
	for some constant $C_1$ that does not depend on $\lambda$.
\end{proposition}
Notice that $\tilde{R}^{\lambda}_M = \tilde{R}_M$ for $r \geq \lambda$.
\begin{proof}
	Using Lemma \ref{lem:4}\ref{lem:4-4} we first obtain
	\begin{align*}
		R_{(D, u^2v_{\lambda}^2g)}  &= u^2 \Biggl[R_{(M, v_{\lambda}^2 g)} - 
						 \frac{u'}{u} v_{\lambda}^2 g \wedge \left(2 \Hess^{v_{\lambda}^2 g}(r) + \frac{u'}{u} 
								\sqnorm{dr}_{v^2_{\lambda}g} v^2_{\lambda}g \right) \\
				&\ghost{= u^2 \Biggl[}\,- 2 \left( \frac{u''}{u} - 2 \left(\frac{u'}{u}\right)^2 \right) v^2_{\lambda}g \wedge dr^2 \Biggr].
	\end{align*}
	Since $\sqnorm{dr}_{v^2_{\lambda}g} = q^{-2}$ and
	\[
		\Hess^{v^2_{\lambda}g}(r) = \frac{1}{r} v^2 g_r + q^{-1} \left( 
				dq(\partial_r)	\,v^2 g - dq \odot dr \right),
	\]
	where we used the orthogonal decomposition $v^2 g = dr^2 + v^2 g_r$ and Lemma \ref{lem:4}\ref{lem:4-3} ($\partial_r := \nabla^{v^2g}r$ is used here and henceforth 
	as an abbreviation), we get
	\begin{equation}
		\label{eq:29}
		\begin{split}
			\frac{u'}{u} v_{\lambda}^2 g &\wedge \left(2 \Hess^{v_{\lambda}^2 g}(r) + \frac{u'}{u} 
				\sqnorm{dr}_{v^2_{\lambda}g} v^2_{\lambda}g \right) \\
			&= v_{\lambda}^2 g \wedge \left( - \frac{2 \alpha}{r^2} v^2 
				g_r + \frac{\alpha^2}{r^2} v^2 g\right) - \frac{2 \alpha}{r} q^{-1}v_{\lambda}^2 g \wedge \left(  
				\partial_r  q \, v^2 g - dq	\odot dr \right) \\
			&= - \frac{\alpha (2 - \alpha)}{r^2} v^2 v_{\lambda}^2 g_r \wedge g_r + \frac{2 \alpha (\alpha - 1)}{r^2} 
				v_{\lambda}^2 g_r \wedge dr^2 \\
			&\ghost{=} \, - \frac{2 \alpha}{r} q^{-1} v_{\lambda}^2 g \wedge \left( \partial_r q  \, v^2 g -  dq \odot dr \right)  
		\end{split}
	\end{equation}
	and
	\begin{equation}
		\label{eq:30}
		 2 \left( \frac{u''}{u} - 2 \left(\frac{u'}{u}\right)^2 \right) v^2_{\lambda}g \wedge dr^2  = - \left( \frac{2 \alpha'}{r}
		 	+ \frac{2\alpha (\alpha - 1)}{r^2} \right) v_{\lambda}^2 g_r \wedge dr^2
	\end{equation}	
	Combining \eqref{eq:29} and \eqref{eq:30}, we get
	\begin{align*}
		R_{(D, u^2 v_{\lambda}^2g)} &= u^2 \biggl[ R_{(M, v_{\lambda}^2g)} + \frac{\alpha(2-\alpha)}{r^2} v^2 v_{\lambda}^2 g_r \wedge g_r + \frac{2\alpha'}{r}
			v_{\lambda}^2 g_r \wedge dr^2 \\
			&\ghost{= u^2 \biggl[} + \frac{2\alpha}{r} q^{-1} v_{\lambda}^2 g \wedge ( \partial_r  q \, v^2 g - dq \odot dr ) \biggr].
	\end{align*} 
	Switching to the algebraic pullback requires multiplying with $(uq)^{-4}$, replacing $v^4 g_r \wedge g_r$ by $R_{S^{n-1} \times \reals}$, $dr$ by 
	$e_n^{\flat}$ and $v^2g$ by $g_{\reals^n}$ and hence directly leads to \eqref{eq:28}, where the error term is given explicitly by
	\[
		E^{\lambda}(\nu, r) := \left[ q^{-1} g_{\reals^n} \wedge \left( \partial_r q \, g_{\reals^n} - dq \odot e_n^{\flat} \right) \right] \circ \exp_{p}^{v^2g}(r 
				\nu).
	\]			
	\par
	Since $\norm{A \wedge B} \leq c \norm{A}\norm{B}$ for selfadjoint operators $A$, $B$, and $\frac{1}{4} \leq q \leq 4$, in order to get the desired bound on 
	$\norm{E^{\lambda}}$, we are left with finding a suitable bound of $\norm{dq}_{v^2g}$ which does not depend on $\lambda$. Recall that
	$q = \frac{v_{\lambda}}{v} = \sqrt{v^{-2} + \phi_{\lambda} (1 - v^{-2})}$. We compute
	\[
		dq(X) = \frac{1}{2q} \left( (1 - \phi_{\lambda}) X(v^{-2}) + X(\phi_{\lambda})(1-v^{-2}) \right).
	\]
	Since $\left|\partial_r \phi_{\lambda}\right| \leq \frac{c}{\lambda}$ for some $c$ and, similarly, $\left|1 - v^{-2}\right| \leq Cr$, we readily obtain
	\[
			\left| \partial_r (\phi_{\lambda})(1-v^{-2})\right| \leq \begin{cases}
						cC, & \text{ if } r < \lambda, \\
						0, & \text{ if } r \geq \lambda,
				\end{cases}
	\]
	for $\phi_{\lambda} \equiv 0$, if $r \geq \lambda$. Because of $X(\phi_{\lambda}) = 0$, if $X \bot \partial_r$, this yields
	\[
		\norm{dq}_{v^2g} \leq \frac{1}{8}\left(\norm{dv^{-2}}_{v^2 g} + cC \right)
	\]
	on $\set{x}{d^{v^2g}(x,p) < \epsilon'}$. This upper bound is evidently independent of $\lambda$.
\end{proof}


	



\subsection{Construction of the deforming functions}

Before the deforming functions can be described, some parameters - which solely depend on the geometry of $(M, g)$ around $p$ - have to be fixed. The various choices being made will become clear during the course of the proof.

The inner cone condition of $C$ with respect to $R_{S^{n-1} \times \reals}$ implies the existence of a $\rho > 0$ so that
\begin{equation}
	\label{eq:39}
	\ball{\lambda \rho}{\tilde{R}_M(\nu, r) + \lambda R_{S^{n-1} \times \reals}} \subset C
\end{equation}
for all $(\nu, r) \in T_p^1M \times (0, \epsilon')$, $\lambda > 0$. \\
We require the starting radius of the bending process $r_0$ to obey $r_0 < \min\left\{\frac{\rho}{4C_1}, \epsilon' \right\}$, where $C_1$ is the constant given by Proposition \ref{prop:10}. 


\subsection{Step 1: Initial bending}

The aim of the initial bending is to prescribe $\alpha$ on a (tiny) interval $[r_1, r_0]$, which will affect the geometry of the manifold in an annular region $\{ r_1 \leq r(x) \leq r_0 \} \subset V$ in such a way as to maintain $\tilde{R}_D \in C$ while making $\alpha$ positive at $r_1$. That $(M, g)$ satisfies $C$ allows us to find $\epsilon > 0$ such that for any isometry $\iota : \reals^n \to T_xM$, $x \in \overline{V}$, we have
\[
	\ball{\epsilon}{\iota^*R_{(M,g)}(x)} \subset C.
\]
Using the star-shapedness of $C$ it therefore suffices to achieve
\[
	\tilde{R}_D \in \ball{\frac{\epsilon}{u^2}}{\frac{1}{u^2} \tilde{R}_M},
\]
which in turn holds provided the right side of
\begin{align*}
	\norm{\tilde{R}_D  - \frac{1}{u^2} \tilde{R}_M} &\leq (uq)^{-2} \left[ 
				\frac{\alpha(2-\alpha)}{r^2} + \frac{2 | \alpha' |}{r} + \frac{2\alpha}{r} \right] C \\
			&\leq \frac{32}{u^2} \left[ \frac{\alpha}{r_1^2} + \frac{| \alpha' |}{r_1} + \frac{\alpha}{r_1} \right] C,
\end{align*}
where $C := \max\left\{ \norm{R_{S^{n-1} \times \reals}}, \norm{g_{\reals^n} \wedge e_n^{\flat} \otimes e_n^{\flat}},\- C_1 \right\}$, is smaller than $\frac{\epsilon}{u^2}$. This is obviously true, if $\alpha \leq \tau$ for some small $\tau > 0$ and if a suitable bound for $\alpha'$ holds on $[r_1, r_0]$. Now $\alpha : [r_1, r_0] \to [0, \tau]$ can easily be prescribed such that $\alpha$ is constant near $r_1$ and $r_0$, being constant with values $\tau$ and $0$, respectively, in these neighborhoods.


\subsection{Step 2: Main bending}

The next step is to produce an extension of $\alpha$ to an interval $[r_2, r_1]$, $r_2 > \lambda$, such that $\alpha$ equals $1$ on a neighborhood around $r_2$ while maintaining $\tilde{R}_D(\nu, r) \in C$ for these $r \in [r_2, r_1]$, $\nu \in T^1_pM$. We derive the following sufficient condition for this:

\begin{lemma}
	\label{lem:7}
	There exists $c > 0$, which does not depend on $\lambda$, such that if $\alpha : (\lambda, r_1] \to [0,1]$, $\alpha'(r) \leq 0$, fulfills
	\begin{equation}
		\label{eq:20}
		\alpha'(r) + c \frac{\alpha(r)(2- \alpha(r))}{r} \geq 0,
	\end{equation}
	then $\tilde{R}_D(\nu, r) \in C$
\end{lemma}

\begin{proof}
	Due to \eqref{eq:39} and the star-shapedness of $C$, the relation $\tilde{R}_D(\nu, r) \in C$ holds, if
	\begin{equation}
		\label{eq:25}
		\tilde{R}_D(\nu, r) \in \ball{\rho \frac{\alpha(2 - \alpha)}{r^2} \left(u q\right)^{-2}}
			{\frac{1}{u^2} S},
	\end{equation}
	where $S := \tilde{R}_M(\nu, r) + \frac{\alpha(2-\alpha)}{r^2} q^{-2} R_{S^{n-1} \times \reals}$.
	By Proposition \ref{prop:10} we have
	\[
		\norm{\tilde{R}_D(\nu, r) - \frac{1}{u^2}S} \leq \left(u q\right)^{-2} \left(\frac{2|\alpha'|}{r} \norm{g_{\reals^n} 
					\wedge e_n^{\flat} \otimes e_n^{\flat}} + \frac{2\alpha}{r} \norm{E^{\lambda}}\right).
	\]
	Hence condition \eqref{eq:25} will hold, if the latter expression is smaller than 
	\[
		\rho \frac{\alpha (2 - \alpha)}{r^2} \left(u q\right)^{-2},
	\]
	which in turn is satisfied, if
	\begin{equation}
		\label{eq:26}
		\norm{E^{\lambda}} < \frac{\rho}{4r} 
	\end{equation}
	and
	\begin{equation}
		\label{eq:27}
		|\alpha'| \leq \frac{\rho}{4} \norm{g_{\reals^n} \wedge e_n^{\flat} \otimes e_n^{\flat}}^{-1} \frac{\alpha(2-\alpha)}{r}
	\end{equation}
	are true. \eqref{eq:26} holds because of $r < r_0 < \frac{\rho}{4 C_1}$. If we choose $c < \frac{\rho}{4} \norm{g_{\reals^n} \wedge e_n^{\flat} \otimes 
	e_n^{\flat}}^{-1}$, \eqref{eq:20} implies \eqref{eq:27}, which proves the lemma.
\end{proof}

So we have to solve a differential equation:

\begin{lemma}
	Given $\tau \in (0,1]$, there exists $r_2 \in (0, r_1)$ and a monotone, non-increasing  $\alpha : (0, r_1] \to [\tau 
	,1]$ fulfilling \eqref{eq:20} such that $\alpha$ is constant on a neighborhood of $r_1$, with value $\tau$, and $\rline{\alpha}{(0, r_2]} \equiv 1$. 
	Moreover, if $\alpha$ is extended to an interval $(0, r_0]$ as described above, there exists $\delta > 0$ such that for any $\gamma \in (0, \delta)$ the function 
	$\alpha$ (and $r_2$) can be chosen in such a way as to ensure that the solution $u$ of \eqref{eq:23} is given by
	\[
		u(r) = \frac{\gamma}{r}
	\]
	for $r < r_2$.
\end{lemma}

Note that $r_2$ depends only on $c$ and $\tau$.

\begin{proof}
	We use the same transformation that was used in \cite{MicallefWang} to tackle the corresponding problem there. So set $r(s) := 
	r_1 e^{-\frac{s}{c}}$ and consider $\beta(s) := \alpha(r(s))$. Then the existence of $\alpha$ as claimed in the statement is 
	easily seen to be equivalent to the existence of a non-decreasing function $\beta : [0, s_2] \to [\tau, 1]$ satisfying
	\begin{equation}
		\label{eq:40}
		\beta'(s) \leq \beta(s)(2- \beta(s)),
	\end{equation}
	$\beta$ being constant on a neighborhood of $0$ with value $\tau$, and $\rline{\beta}{[s_2, \infty)} \equiv 1$, where $s_2 := s(r_2)$.
	\par
	The construction of $\beta$ is easily accomplished: Take any non-decreasing function $\beta$ matching the named side conditions together with the 
	grow restriction $\beta'(s) \leq \tau (2 - \tau)$. Then \eqref{eq:40} automatically holds simply because $x \mapsto 
	x(2-x)$, restricted to $x \in [\tau, 1]$, has a minimum at $x = \tau$.
	\par
	For the second statement observe that, since $u(r_0) = 1$, we have
	\[
		u(r) = \exp \left( - \int_{r_0}^r \frac{\alpha(t)}{t}dt\right).
	\]
	In particular for $r < r_2$
	\[
		u(r) = \frac{r_2}{r} \exp \left( \int_{r_2}^{r_1} \frac{\alpha(t)}{t} \, dt\right) \exp \left( \int_{r_1}^{r_0} 
							\frac{\alpha(t)}{t} \, dt\right)
	\]
	The identity $ \int_{r_2}^{r_1} \frac{\alpha(t)}{t}dt = \frac{1}{c} \int_0^{s_2} \beta(s) \, ds$ yields
	\[
		\gamma =  \exp \left( \int_{r_1}^{r_0} \frac{\alpha(t)}{t}dt\right) r_1 \exp \left(\frac{1}{c} \int_0^{s_2} (\beta(s) - 1) ds 
			\right)
	\]
	Thus, by suitably arranging $\beta$ (and $s_2 = s(r_2)$) any desired value for $\gamma$ in a range $(0, \delta)$, 
	$\delta := r_1 \exp \left( \int_{r_1}^{r_0} \frac{\alpha(t)}{t}dt\right)$, can be achieved.
\end{proof}


\subsection{Step 3: Smoothing of the end}

In the final step the yet unspecified parameter $\lambda \in (0, r_2)$ has to be chosen.

\begin{lemma}
	\label{lem:10}
	There exists a constant $C_2 > 0$ depending only on $(M, g)$, $V$ and $p$ so that
	\[
		\norm{\tilde{R}_M^{\lambda}(\nu, r)} \leq \frac{C_2}{\lambda}
	\]
	for $\lambda < r_0$.
\end{lemma}

\begin{proof}
	We only need to consider the region $\{ r \leq \lambda \}$, because we have $R_{(M, v_{\lambda}^2g)} = R_{(M,g)}$ outside this 
	region. Applying Theorem $1.159b$ of \cite{Besse} we get
	\[
		R_{(M, v_{\lambda}^2g)} = - q^2 \left[2v^2g \wedge \left( 
			\Hess^{v^2g}(\log(q)) - \left( d \log(q)\right)^2 + \frac{1}{2} \sqnorm{d 
			\log(q)}_{v^2g} v^2g \right) \right]
	\]
	This formula together with the rule $\norm{A \wedge B} \leq c \norm{A}\norm{B}$ for selfadjoint operators $A, B$ immediately 
	yields the desired result once we have the estimates
	\begin{equation}
		\label{eq:21}
		\norm{d v_{\lambda}}_{v^2 g} = \BigO\left(1\right), \quad
		\norm{\Hess^{v^2g} (v_{\lambda})}_{v^2g} = \BigO\left(\lambda^{-1}\right).
	\end{equation}
	The first estimate was already derived in the proof of Proposition \ref{prop:10}. For the second one we make use of the 
	identity $\Hess^{v^2g} (v_{\lambda})(X,Y) = X\left(Yv_{\lambda}\right) - \nabla^{v^2g}_XY \left( v_{\lambda}\right)$ and compute
	\begin{align*}
		\partial_r^2v_{\lambda} &= - \frac{1}{2v_{\lambda}^2} \partial_r v_{\lambda} \left( (\partial_r \phi_{\lambda}) 
						(v^2-1) + \phi_{\lambda} \partial_r (v^2-1) \right) + \\
				&\ghost{=} \, \frac{1}{2v_{\lambda}} \left( \partial_r^2 \phi_{\lambda} (v^2-1) + 2 \partial_r \phi_{\lambda} \partial_r
						(v^2-1) + \phi_{\lambda} \partial_r^2(v^2-1) \right).
	\end{align*}
	Because of $v^2-1 = \BigO(r)$, $r \leq \lambda$ and $\partial_r^k\phi_{\lambda} = \BigO\left(\lambda^{-k}\right)$ we get
	\[
		\partial^2_r v_{\lambda} = \BigO\left(\lambda^{-1}\right).
	\]
	Similarly, if $X \bot \partial_r$, by recalling that $\phi_{\lambda}$ is a radial function we obtain
	\begin{align*}
		\partial_r \left(Xv_{\lambda}\right) &= -\frac{1}{2v_{\lambda}^2} \partial_r v_{\lambda} \left( \phi_{\lambda} X(v^2-1) \right) +\\
				&\ghost{=} \frac{1}{2v_{\lambda}} \left( \partial_r \phi_{\lambda} X(v^2-1) + \phi_{\lambda} \partial_r X(v^2-1) \right).
	\end{align*}
	This gives us $\partial_r \left(Xv_{\lambda}\right) = \BigO\left(\lambda^{-1}\right)$ by the same reasoning. Finally, for $X, Y \bot 
	\partial_r$, we easily get $X\left(Yv_{\lambda}\right) = \BigO(1)$. This proves the hessian estimate in \eqref{eq:21}.
\end{proof}

For $r \leq r_2$ Proposition \ref{prop:10} yields
\[
	\tilde{R}_D(\nu, r) = u^{-2} \tilde{R}_M^{\lambda}(\nu, r) + \left(uq\right)^{-2} \left[ \frac{1}{r^2} R_{S^{n-1} \times 
			\reals} + \frac{2}{r} E^{\lambda}(\nu, r) \right],
\]
since $\alpha|_{(0, r_2]} \equiv 1$. For $r \in [\lambda, r_2]$, $\tilde{R}_D \in C$ continues to hold by Lemma \ref{lem:7}. For $r < \lambda$, again it is enough to guarantee the inclusion condition
\[
	\tilde{R}_D(\nu, r) \in \ball{\rho \frac{1}{r^2} \left(uq\right)^{-2}}
		{\frac{1}{u^2}\left(\tilde{R}_M(\nu, r) + \frac{1}{r^2q^2} R_{S^{n-1} 
			\times \reals}\right)}
\]
in order to keep the curvature tensor inside $C$. Thus by Lemma \ref{lem:10} the analogous estimate reads
\begin{align*}
	\norm{\tilde{R}_D(\nu, r) - \frac{1}{u^2} \left(\tilde{R}_M + \frac{1}{r^2 q^2} R_{S^{n-1} \times \reals}\right)} \leq
			\frac{1}{u^2} \left(\norm{\tilde{R}_M} + \frac{2C_1}{q^2r} + \frac{C_2}{\lambda} \right).
\end{align*}
This is smaller than $\rho \frac{1}{(ruq)^2}$ provided
\[
	\lambda < \min\left\{ r_2, \rho^{\frac{1}{2}}\left(48 \max \norm{\tilde{R}_M}\right)^{-\frac{1}{2}},
		\frac{\rho}{6 C_1}, \frac{\rho}{48C_2} \right\}.
\]
Choosing $\lambda$ such that these constraints are met is possible since none of the expressions of the right side depend on $\lambda$. By doing so, the third step of the deformation construction is completed, and so is the proof of Theorem \ref{thm:8}.

\end{document}